\documentclass[journal]{IEEEtran}

\usepackage{amsmath,amssymb,amsfonts,amsthm,wasysym}
\usepackage{algorithmic}
\usepackage[boxed]{algorithm2e}
\usepackage{mathtools}
\usepackage[noadjust]{cite}
\usepackage{url}
\usepackage{hyperref}
\hypersetup{
     colorlinks = true,
     linkcolor = blue,
     anchorcolor = blue,
     citecolor = blue,
     filecolor = blue,
     urlcolor = blue
     }
\usepackage{euscript}
\usepackage{fontawesome}
\usepackage[bb=boondox]{mathalfa}
\usepackage{stmaryrd}
\usepackage{bm}
\usepackage{booktabs} 
\usepackage{multirow}
\usepackage{threeparttable}
\usepackage{balance}

\usepackage{pgfplots}
\usepackage{tikz}
\pgfplotsset{compat=1.17}
\usetikzlibrary{positioning}
\usepgfplotslibrary{fillbetween}
\usetikzlibrary{calc}
\usetikzlibrary{angles, quotes}

\newtheorem{theorem}{Theorem}

\theoremstyle{remark}

\newtheorem{remark}{Remark}

\newcommand{\fl}[1]{\mathrm{#1}}
\newcommand{\minimize}[1]{\underset{{#1}}{\text{min}}}
\newcommand{\maximize}[1]{\underset{{#1}}{\text{max}}}
\newcommand{\braceit}[1]{\left({#1}\right)}

\newcommand{\st}{\text{s.t.}}

\newcommand{\norm}[1]{\left\lVert{#1}\right\rVert}
\DeclareMathSymbol{\shortminus}{\mathbin}{AMSa}{"39}
\def\shortplus{
\begin{tikzpicture}
    \node[draw=none,scale=0.5,inner sep=0] at (0,0) {+};
\end{tikzpicture}
}
\newcommand\ncoverline[1]{\mkern1mu\overline{\mkern-1mu#1\mkern-1mu}\mkern1mu}

\definecolor{maincolor}{HTML}{000000}
\definecolor{blue}{RGB}{31,64,122}
\definecolor{red}{HTML}{A52A2A}
\definecolor{mygreen}{HTML}{228B22}

\usepackage{nomencl}
\makenomenclature

\usepackage{relsize}
\newcommand*{\mydot}{\mathrel{\mathsmaller{\mathsmaller{\odot}}}}

\begin{document}
\begingroup
\allowdisplaybreaks

\title{Multi-Stage Decision Rules for Power Generation \& Storage Investments with Performance Guarantees}

\author{Vladimir~Dvorkin,~\IEEEmembership{Member,~IEEE,}
        Dharik~Mallapragada 
        and~Audun~Botterud,~\IEEEmembership{Member,~IEEE,} 
\thanks{Vladimir Dvorkin and Dharik Mallapragada are with the Energy Initiative and Audun Botterud is with the Laboratory for Information \& Decision Systems of Massachusetts Institute of Technology, Cambridge,
MA, 02139, USA. E-mail: \{dvorkin,dharik,audunb\}@mit.edu. The authors acknowledge funding from the MIT Energy Initiative Future Energy Systems Center. Vladimir Dvorkin was also supported in part by by the Marie Sklodowska-Curie Actions and Iberdrola Group, Grant Agreement \textnumero101034297 – project Learning ORDER.}
}
\maketitle

\begin{abstract}
We develop multi-stage linear decision rules (LDRs) for dynamic power system generation and energy storage investment planning under uncertainty and propose their chance-constrained optimization with performance guarantees. First, the optimized LDRs guarantee operational and carbon policy feasibility of the resulting dynamic investment plan even when the planning uncertainty distribution is ambiguous. Second, the optimized LDRs internalize the tolerance of the system planner towards the stochasticity (variance) of uncertain investment outcomes. They can eventually produce a quasi-deterministic investment plan, which is insensitive to uncertainty (as in deterministic planning) but robust to its realizations (as in stochastic planning). Last, we certify the performance of the optimized LDRs with the bound on their sub-optimality due to their linear functional form. Using this bound, we guarantee that the preference of LDRs over less restrictive -- yet poorly scalable -- scenario-based optimization does not lead to financial losses exceeding this bound. 
We use a testbed of the U.S. Southeast power system to reveal the trade-offs between the cost, stochasticity, and feasibility of LDR-based investments. 
We also conclude that the LDR sub-optimality depends on the amount of uncertainty and the tightness of chance constraints on operational, investment and policy variables.
\end{abstract}


\IEEEpeerreviewmaketitle

\section*{Nomenclature}
\addcontentsline{toc}{section}{Nomenclature}
\subsection{Model dimensions}

\begin{IEEEdescription}[\IEEEusemathlabelsep\IEEEsetlabelwidth{$TT$}]
\item[$E$] Number of transmission lines, indexed by $e$
\item[$H$] Number of representative hours, indexed by $h$
\item[$N$] Number of network nodes, indexed by $i$
\item[$\Omega$] Number of uncertainty scenarios, indexed by $\omega$
\item[$T$] Number of investment stages, indexed by $t$
\item[$W$] Number of operating horizons, indexed by $w$
\end{IEEEdescription}

\subsection{Parameters}
\begin{IEEEdescription}[\IEEEusemathlabelsep\IEEEsetlabelwidth{$TTt$}]
\item[$b_{t}$] Stage-specific investment budget [\$]
\item[$c_{t}^{\mydot}$] Vectors of operating costs [\$/MWh] 
\item[$e^{\mydot}/\overline{e}_{t}$] Emission intensity / annual cap \![t\! CO$_2$/MWh \;/ t\! CO$_2$]
\item[$\overline{f}$] Vector of transmission line capacities [MW]
\item[$F$] Matrix of power transfer distribution factors [\%]
\item[$k_{twh}^{\mydot}$] Vectors of capacity factors [p.u.] 
\item[$\ell_{t}$] Vector of annual peak loads [MW]
\item[$o_{t}^{\mydot}$] Vectors of O\&M cost [\$/MW-year] 
\item[$\overline{p}_{t}$] Vector of existing generation capacities [MW] 
\item[$q_{t}^{\mydot}$] Vectors of investment cost [\$/MW] 
\item[$r^{\mydot}$] Up- and down-regulation capacity vectors [\%] 
\item[$\eta^{\mydot}$] Storage charging/discharging efficiency [\%]
\item[$\omega_{w}$] Annual weight of operating horizon $w$ [-]  
\item[$\overline{\mydot}^{\fl{max}}$] Maximum investment limit for technology $\mydot$ [MW]
\item[$\Sigma,\overline{\Sigma}$] Covariance matrix and its Cholesky factorization
\item[$\varepsilon^{\mydot},\overline{\varepsilon}^{\mydot}$] Joint and individual constraint violation probabilities
\end{IEEEdescription}

\subsection{Vectors of Decision Variables}
\begin{IEEEdescription}[\IEEEusemathlabelsep\IEEEsetlabelwidth{$TTt$}]
\item[$x_{t}^{\mydot}/z_{t}^{\mydot}$] Auxiliary variables for the distributionally robust reformulation of chance constraints [MW or MWh]
\item[$\overline{y}_{t}$] Candidate generation capacity [MW] 
\item[$\overline{\vartheta}_{t}$] Candidate storage energy capacity [MWh]
\item[$\overline{\varphi}_{t}$] Candidate storage dis/charging capacity [MW]
\item[$p_{twh}$] Existing generation dispatch [MWh] 
\item[$y_{twh}$] Candidate generation dispatch [MWh] 
\item[$\vartheta_{twh}$] Candidate storage state of charge [MWh]
\item[$\varphi_{twh}^{\mydot}$] Candidate storage dis/charging rate [MWh]
\end{IEEEdescription}

\subsection{Other Notation}
For a variable vector $x$, its upper-case counterpart $X$ denotes a coefficient matrix of the corresponding linear decision rule, e.g., for candidate generation capacity we will have $\overline{y}_{t}$ and $\overline{Y}_{t}$. Symbol $\circ$ stands for the Schur product. $\mathbb{0}$ and $\mathbb{1}$ denote vectors (matrices) of zeros and ones, respectively. Operator $|\!\cdot\!|$ is the absolute value operator, $\norm{\cdot}$ is the Euclidean norm and $\text{Tr}[\cdot]$ is the matrix trace.  The ordered set $1,\dots,N$ is denoted by $\llbracket N\rrbracket$. Superscript $+/-$ stands for charging/discharging.

\section{Introduction}

\IEEEPARstart{W}{e} study the co-optimization of generation and energy storage investments to guarantee a reliable, least-cost, and carbon policy-feasible future energy supply. This problem is challenging at least for two reasons. First, the investments are affected by planning uncertainty which spans long, multi-stage investment horizons, as illustrated in Fig. \ref{fig:capex}. Second, the investment decisions must be informed by many instances of the optimal power flow (OPF) problems with high spatial and temporal resolution, i.e., to capture intra- and inter-annual variations of renewable generation and electricity demand across large geographical areas \cite{conejo2016investment}. Therefore, such co-optimization tends to demonstrate high computational complexity.

\begin{figure}
\centering
\includegraphics[width=0.48\textwidth]{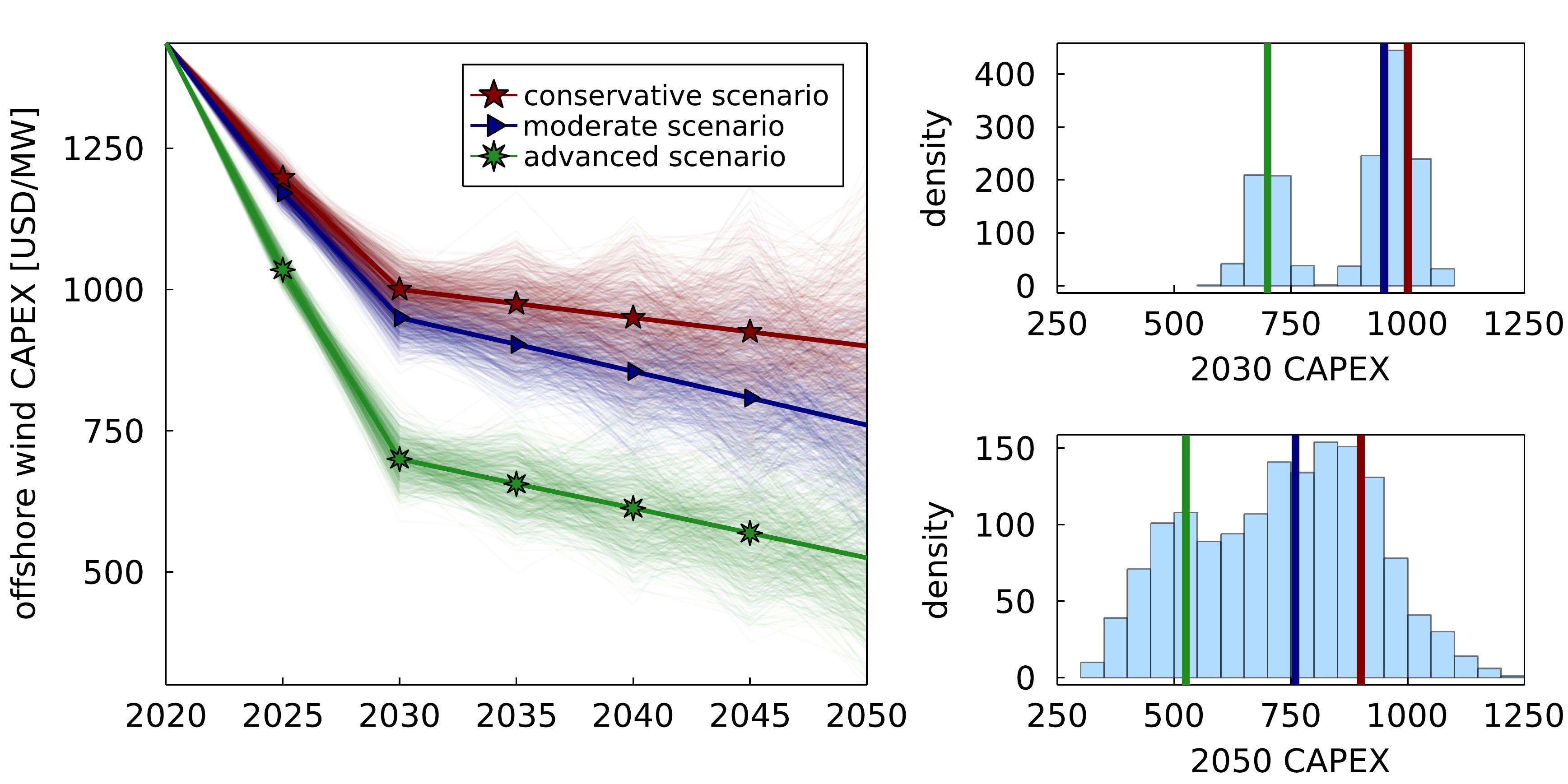}
\caption{Uncertainty of offshore wind CAPEX. Three thick lines depict baseline CAPEX scenarios from the 2021 NREL Annual Technology Baseline (ATB) \cite{vimmerstedt2021annual}, and the thin lines display 1500 uncertainty realization scenarios. This stochastic process is estimated using 2016--2021 data and assuming Log-Normal error distribution. The histograms on the right depict CAPEX probability densities for 2030 and 2050, where the vertical colored lines depict the baseline scenarios. Notably, the distribution shape, statistical moments and support substantially change throughout the planning horizon.}
\label{fig:capex}
\end{figure}

One solution to this problem is to discretize the planning uncertainty distribution using scenarios \cite{jin2011modeling}. Solving the scenario approximation to optimality, however, requires including at least $N$ in-sample scenarios, where $N$ is moderate for two-stage problems but prohibitively large ($\gg10^6$) for multi-stage investment programs \cite{shapiro2005complexity}. Notably, using fewer in-sample scenarios to achieve computational tractability often yields a poorer out-of-sample performance \cite{wiesemann2014distributionally}. To respect the in-sample requirement, one employs decomposition techniques, including dual decomposition \cite{liu2017multistage,dvorkin2018consensus} or stochastic dual dynamic programming \cite{zou2016nested}, while facing their intrinsic trade-offs between computational tractability and solution optimality.  

A scenario-free approximation of the stochastic investment problem is thus an appealing alternative. Velloso \textit{et al.} \cite{velloso2020distributionally} developed such an approximation that models the deviations from the baseline scenarios (as in Fig. \ref{fig:capex}) using the mean and distribution support information. In a similar distributionally robust manner, Pourahmadi \textit{et al.} \cite{pourahmadi2019distributionally} accommodated operational uncertainty of renewable generation assets. Their models, however, are limited to two-stage stochastic programming targeting the final-stage uncertainty (e.g., 2050 in Fig. \ref{fig:capex}), hence disregarding economic, operational and policy rationale of investments' recourse as uncertainty gradually resolves. 

The key to scenario-free approximations of the dynamic, multi-stage investment problems is in linear decision rules (LDRs) that substitute recourse decisions with linear functions of uncertain parameters, thus substantially reducing the problem complexity \cite{shapiro2005complexity,delage2015robust}. Specifically, if scenario-based stochastic programs grow exponentially in the number of investment stages, the LDR-based formulations grow only linearly. Multi-stage LDRs have been thus applied before by Dom{\'\i}nguez \textit{et al.} \cite{dominguez2016investing} to the multi-stage generation investment problem and by Dehghan \textit{et al.} \cite{dehghan2017adaptive} to the multi-stage transmission planning, showing a drastic complexity reduction for realistically scaled problems. 

Despite the progress in \cite{dominguez2016investing,dehghan2017adaptive}, the applications of LDRs to investment planning are limited for three important reasons. First, the LDRs in \cite{dominguez2016investing,dehghan2017adaptive} are optimized on a robust uncertainty set, which is ignorant of the temporal evolution of the planning uncertainty. Moreover, in robust optimization, the optimized investment plan is optimal with respect to the extreme, low-probability uncertainty outcomes, thus resulting in unnecessary conservatism in terms of investment planning costs. While the conservatism of the robust solution can be reduced by optimizing the size of the uncertainty set, e.g., using the adaptive uncertainty set optimization from \cite{liang2020robust}, it requires an additional optimization layer to identify the secure and cost-optimal size of the uncertainty set.
Second, restricting recourse decisions to linear functions makes them highly sensitive to uncertainty realizations \cite{dvorkin2021stochastic,dvorkin2022multi}. 
As a consequence, LDRs are likely to produce highly uncertain decarbonization pathways, hence providing less informed policy advice. 
Finally, although the LDR approximation is optimal for some rare optimization structures \cite{georghiou2021optimality}, in the context of large-scale capacity expansion it is likely to be sub-optimal with respect to the scenario approximation that makes no recourse assumptions. Consequently, the LDR sub-optimality translates into financial losses, reducing the attractiveness of LDR applications in practice.

\subsection{Contributions} 

Motivated by dynamic planning uncertainties and by missing performance guarantees of existing LDR-based approaches, we develop a new LDR approximation of the stochastic multi-stage generation and energy storage investment planning. We make the following technical contributions:

\textit{1)} We develop a chance-constrained multi-stage LDR optimization, which combines the specificity of stochastic programming (by modeling uncertainty using probability distributions) and robust optimization (by hedging against distributional uncertainty), to immunize the generation and energy storage investment planning against planning uncertainty. This optimization guarantees feasibility of LDR-guided investments up to the given tolerance to investment, engineering and carbon policy constraint violation, and is recast as a second-order cone program solved efficiently using off-the-shelf software. 

\textit{2)} We study investment sensitivities to uncertainty and find that the standard expected cost-minimization postpones investments and accumulates significant variance of the investment plan, making decarbonization pathways highly uncertain. We thus extend the LDR optimization to control investment sensitivities by constraining the variance of investment decision rules. This variance-constrained optimization reveals the trade-off between the cost and uncertainty of the investment plan. By tightening the variance constraints, we obtain the so-called {\it quasi-deterministic} investment plan -- the fixed plan, which is insensitive to uncertainty but robust to its realizations. This plan eventually identifies the optimal adjustments of deterministic investments to immunize them against uncertainty.

\textit{3)} To guarantee the economic performance of investment LDRs, we develop a framework to bound the sub-optimality of the \textit{optimized} investment LDRs, compared to a less restrictive -- yet unattainable at scale -- scenario-based solution, which does not require linear recourse. Using this bound, we guarantee that switching from scenario to LDR approximation of the optimal investments will not incur a loss exceeding this bound on average. To analyze the likelihood of the worst-case LDR sub-optimality, we devise a bilevel optimization problem for small problem instances to learn those uncertainty realizations that cause the largest sub-optimality. 

\subsection{Paper Organization}

Section \ref{sec:from_det_to_sto} explains preliminaries, uncertainty modeling and the chance-constrained investment co-optimization. Section \ref{sec:cc_ldr} introduces the LDR approximation. Section \ref{sec:inv_var_red_and_det_eq} details the variance-constrained optimization, and Section \ref{sec:sub-optimality_gurantees} details the LDR sub-optimality analysis. Section \ref{sec:case_study} demonstrates investment LDRs on a model of the U.S. Southeast's power system. Extended primal and dual problem formulations are relegated to Appendix \ref{app:full_tract_ref} and \ref{app:sto_dual}. Section \ref{sec:conclusions} concludes. All data and code are available in the e-companion in \cite{dvorkin2022}. 

\section{From Deterministic to Stochastic Multi-Stage Power System Investment Planning}\label{sec:from_det_to_sto}
\subsection{Deterministic Problem Formulation}
We consider a $T-$stage investment horizon, where each stage models system operations using a sequence of $W$ operating horizons, each including $H$ representative hours, as depicted in Fig. \ref{fig:planning_scheme}. We thus solve $T\times W$ multi-period OPF problems to identify the feasible and least-cost set points for all generators and energy storage systems. OPF problems and investment decisions are co-optimized by minimizing the planning cost function:
\begin{subequations}\label{prog:base}
\begin{align}
\minimize{\mathcal{V}}\quad&\textstyle\sum_{t=1}^{T}\bigg(\underbrace{q_{t}^{\text{g}\top}\overline{y}_{t} + q_{t}^{\text{s}\top}\overline{\vartheta}_{t} + q_{t}^{\text{p}\top}\overline{\varphi}_{t}}_{\text{annualized investment cost}} + \nonumber\\
&\quad \underbrace{o_{t}^{\text{e}\top}\overline{p}_{t} +  \textstyle\sum_{\tau=1}^{t}\Big(o_{t}^{\text{c}\top}\overline{y}_{\tau} + o_{t}^{\text{s}\top}\overline{\vartheta}_{\tau} + o_{t}^{\text{p}\top}\overline{\varphi}_{\tau}\Big)}_{\text{annualized operation and maintenance (O\&M) cost}} + \nonumber \\
    &\quad\quad\underbrace{\textstyle\sum_{w=1}^{W}\omega_{w}\textstyle\sum_{h=1}^{H}\big(c_{t}^{\text{e}\top}p_{twh} + c_{t}^{\text{c}\top}y_{twh}\big)}_{\text{annualized fuel cost}}\bigg),
    \label{eq_obj}
\end{align}
where the first term includes stage--specific investment costs of power generation $(q_{t}^{\text{g}\top}\overline{y}_{t})$, energy storage $(q_{t}^{\text{s}\top}\overline{\vartheta}_{t})$ and charging/discharging power $(q_{t}^{\text{p}\top}\overline{\varphi}_{t})$, the second term models stage--specific O\&M costs of existing and built candidate units, and the last term models stage-specific total fuel costs weighted by weight $\omega_{w}$ of the operating horizon $w$ in the year.
Although the linear investment and O\&M costs can be combined, we still differentiate them as they can be affected by uncertainties differently.
The set $\mathcal{V}=\{\overline{y},y,p,\overline{\vartheta},\vartheta,\overline{\varphi},\varphi^{\shortplus},\varphi^{\shortminus}\}$ includes investment and operational variables, subject to engineering, policy and investment constraints. 

The following OPF constraints
\begin{align}
&\mathbb{1}^{\top}\braceit{p_{twh} + y_{twh} + \varphi_{twh}^{\shortminus} - k_{twh}^{\ell}\circ\ell_{t} - \varphi_{twh}^{\shortplus}} = 0,\label{prog:base_eq}\\
&\big|F\braceit{p_{twh} + y_{twh} + \varphi_{twh}^{\shortminus} - k_{twh}^{\ell}\circ\ell_{t} - \varphi_{twh}^{\shortplus}}\big|\leqslant\overline{f},\label{prog:base_flow_max}
\end{align}
are enforced to ensure the system-wide energy balance, by off-setting the consumption of variable electric loads and storage charging $k_{twh}^{\ell}\circ\ell_{t} + \varphi_{twh}^{\shortplus} $ by dispatched generation and storage discharging $p_{twh} + y_{twh} + \varphi_{twh}^{\shortminus} $ in \eqref{prog:base_eq}, and the satisfaction of the power line limits in \eqref{prog:base_flow_max}, using a PTDF-based DC power flow representation \cite{chatzivasileiadis2018optimization}. Unlike the angle-based power flow model, this representation does not require modeling voltage angles as optimization variables, thus yielding a more compact problem formulation. Generation capacity limits are enforced for the existing and candidate units as
\begin{align}
&\mathbb{0}\leqslant p_{twh} \leqslant k_{twh}^{\fl{e}}\circ\overline{p}_{t},\label{prog:gen_cap_ex}\\
&\mathbb{0}\leqslant y_{twh} \leqslant k_{twh}^{\fl{c}}\circ \textstyle\sum_{\tau=1}^{t}\overline{y}_{\tau},\label{prog:gen_cap_can}
\end{align}
accounting for power factors that amount to $1$ for conventional units and vary between $0$ and $1$ for renewable units. To differentiate generation technologies not only by costs but also by their flexibility, the inequalities
\begin{align}
&-r^{\fl{e}\shortminus}\circ\overline{p}_{t} \leqslant p_{twh} - p_{tw(h-1)} \leqslant r^{\fl{e}\shortplus}\circ\overline{p}_{t},\label{prog:gen_ramp_exist}\\
& y_{twh} - y_{tw(h-1)} \geqslant -r^{\fl{c}\shortminus}\circ\textstyle\sum_{\tau=1}^{t}\overline{y}_{\tau},  \\
& y_{twh} - y_{tw(h-1)} \leqslant r^{\fl{c}\shortplus} \circ \textstyle\sum_{\tau=1}^{t}\overline{y}_{\tau},\label{prog:gen_ramp_cand}
\end{align}
are introduced to maintain inter-temporal generation changes within the maximum up- and down-regulation limits, respectively defined as a percentage $r^{\odot}$ of the installed capacity, which are set to 100\% for renewable energy technologies. The energy storage operation is modeled using equations
\begin{align}
&\vartheta_{twh} = \vartheta_{tw(h-1)} + \varphi_{twh}^{\shortplus}\eta^{\shortplus} - \varphi_{twh}^{\shortminus}/\eta^{\shortminus},\label{prog:stor_sos}
\end{align}
where we assume zero initial state-of-charge $\vartheta_{tw0}=\mathbb{0}$ for all investment stages and representative periods, and enforcing the following set of constraints on the storage state of charge, charging and discharging variables:
\begin{align}
&\mathbb{0}\leqslant \vartheta_{twh} \leqslant \textstyle\sum_{\tau=1}^{t}\overline{\vartheta}_{\tau}, \label{prog:stor_sos_lim}\\
&\mathbb{0}\leqslant\varphi_{twh}^{\shortplus} \leqslant \textstyle\sum_{\tau=1}^{t}\overline{\varphi}_{\tau}, \label{prog:stor_charging}\\
&\mathbb{0}\leqslant\varphi_{twh}^{\shortminus} \leqslant \textstyle\sum_{\tau=1}^{t}\overline{\varphi}_{\tau}, \label{prog:stor_discharging}\\
&\varphi_{twh}^{\shortplus} + \varphi_{twh}^{\shortminus} \leqslant \textstyle\sum_{\tau=1}^{t}\overline{\varphi}_{\tau}.\label{prog:stor_charge_discharge}
\end{align}
To preserve convexity, decision variables in \eqref{prog:stor_charging}--\eqref{prog:stor_charge_discharge} are not constrained to explicitly prevent simultaneous charging and discharging. This modeling choice is not limiting, because we model an aggregated behavior of a large population of asynchronous storage devices, e.g., totaling 10s of GWs in Section \ref{sec:case_study}. Importantly, in our experiments, we also do not observe negative pricing -- the sufficient condition for simultaneous charging/discharging to occur in practice \cite[Proposition 4]{xu2020lagrangian}.
The next modeling choice is to disregard the explicit representation of storage degradation, e.g., as a non-convex function of daily cycles. Instead, the fixed O\&M costs for Li-ion battery systems $o_{t}^{\text{s}}$ and $o_{t}^{\text{p}}$ account for periodic replacements of battery cells to maintain usable capacity throughout the asset life -- the modeling approach suggested by the NREL annual technology baseline studies \cite{vimmerstedt2021annual,xu2022role}.

\begin{figure}
    \centering
    \fbox{
    \resizebox{0.47\textwidth}{!}{%
    \includegraphics{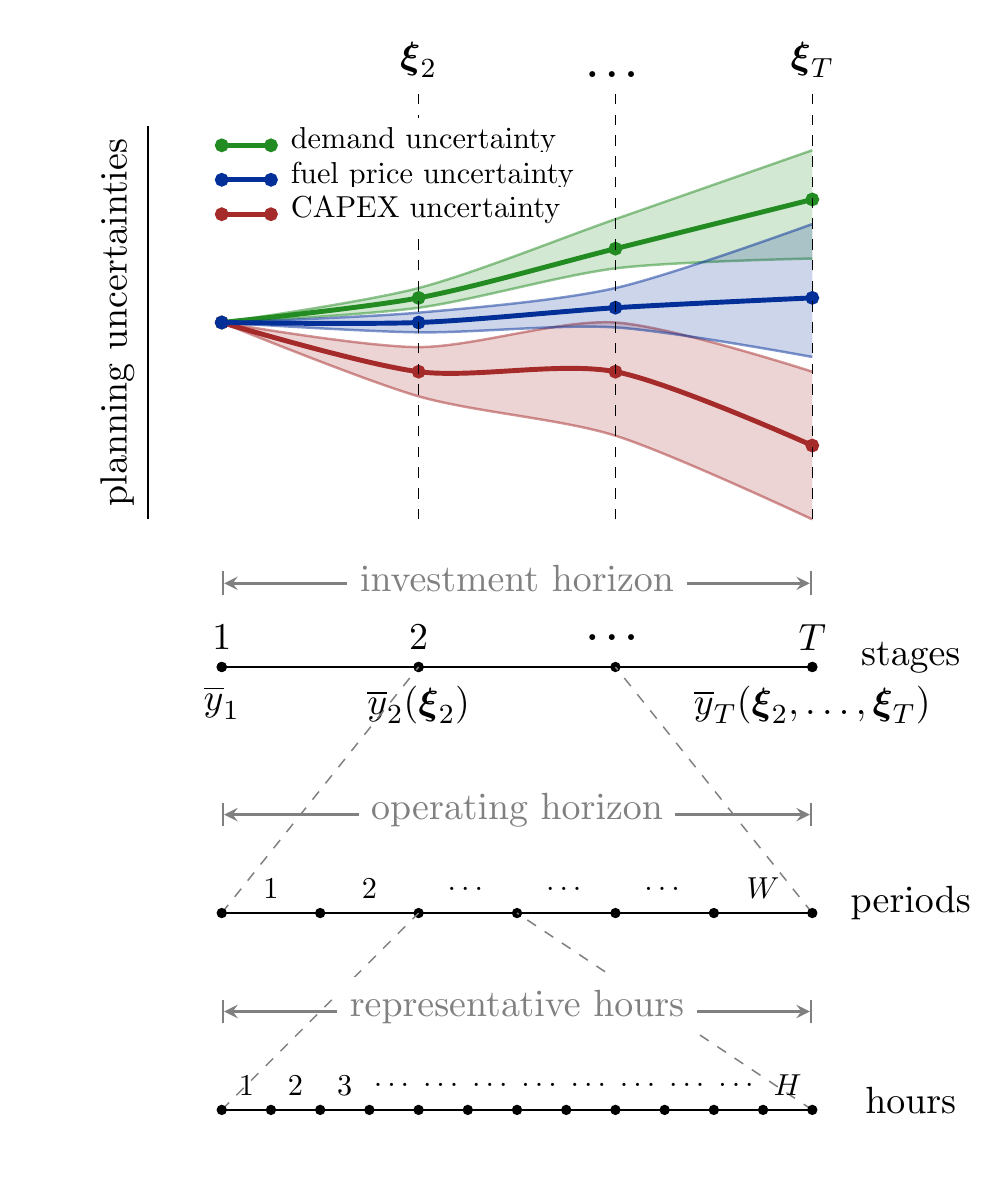}
    }}
    \caption{Decision-making timeline: investment horizon is split into $T$ stages; at any stage, the investment decisions are affected by planning uncertainty realized upon that stage. The operating horizon between the stages is modeled using $W$ representative periods, each consisting of $H$ operating hours.}
    \label{fig:planning_scheme}
\end{figure}

To model the carbon policy constraints, consider the annual CO$_{2}$ limit enforced on power generation as
\begin{align}
\textstyle\sum_{w=1}^{W}\omega_{w}\textstyle\sum_{h=1}^{H}\braceit{e^{\fl{e}\top}p_{twh} + e^{\fl{c}\top}y_{twh}} \leqslant \overline{e}_{t}.\label{prog:policy_cap}
\end{align}
The last constraints model generation and energy storage investment limits (e.g., due to engineering or permitting limits) and stage-specific annualized budget limits, i.e.,
\begin{align}
&
\mathbb{0}\leqslant\overline{y}_{t}\leqslant \overline{y}_{t}^{\fl{max}},\;
\mathbb{0}\leqslant\overline{\varphi}_{t}\leqslant \overline{\varphi}_{t}^{\fl{max}},\;
\mathbb{0}\leqslant\overline{\vartheta}_{t}\leqslant \overline{\vartheta}_{t}^{\fl{max}},\label{prog:inv_lim_1}\\
&
q_{t}^{\text{g}\top}\overline{y}_{t} + q_{t}^{\text{s}\top}\overline{\vartheta}_{t} + q_{t}^{\text{p}\top}\overline{\varphi}_{t} \leqslant b_{t}.
\label{prog:inv_lim_2}
\end{align}
\end{subequations}
Constraints \eqref{prog:base_eq}--\eqref{prog:inv_lim_2} must hold for all stages $t\in\llbracket T\rrbracket$, operating horizons $w\in\llbracket W\rrbracket$ and representative hours $h\in\llbracket H\rrbracket$.

\begin{remark}
For compactness, in formulation \eqref{prog:base} we assume that each node hosts at most one exiting/candidate generator, load, or energy storage unit. In our model implementation in \cite{dvorkin2022}, however, we use incidence matrices for all network components to accommodate arbitrary network configurations.
\end{remark}

\subsection{Modeling Uncertainty}

To model uncertainty in planning optimization data, we express the data as a function of unit-mean random variables $\boldsymbol{\xi}=(\boldsymbol{\xi}_{1},\dots,\boldsymbol{\xi}_{T})\in\mathbb{R}^{n}$, where $n$ is the total number of random variables.  For any investment stage $t$, vector $\boldsymbol{\xi}_{t}\in\mathbb{R}^{n_{t}}$ collects $n_{t}$ random variables revealed {\it at} stage $t$, and vector $\boldsymbol{\xi}^{t}=(\boldsymbol{\xi}_{1},\dots,\boldsymbol{\xi}_{t})\in\mathbb{R}^{n^{t}}$ collects $n^{t}$ random variables revealed {\it up to} stage $t$. Consistency then requires $n=\sum_{t=1}^{T}n_{t}$ and $n^{t}=\sum_{\tau=1}^{t}n_{\tau}$. We also assume that there is no uncertainty at the first stage (hear-and-now) and set $\boldsymbol{\xi}_{1}$ to $1$ ($n_{1}=1$). To extract $\boldsymbol{\xi}^{t}$ from $\boldsymbol{\xi}$, consider truncation matrix $S_{t}\in\mathbb{R}^{n^{t}\times n}$, such that $\boldsymbol{\xi}^{t}=S_{t}\boldsymbol{\xi}$. That is, the left block of matrix $S_{t}$ is an $n^{t}\times n^{t}$ identity matrix, and the remaining elements of $S_{t}$, if any, are zeros. The uncertain data is then defined as
\begin{align}\label{eq:uncertain_data_def}
    \begin{aligned}
    &q_{t}^{\mydot}(\boldsymbol{\xi}^{t}) = Q_{t}^{\mydot}S_{t}\boldsymbol{\xi}, \;\;  \ell_{t}(\boldsymbol{\xi}^{t}) = L_{t}S_{t}\boldsymbol{\xi},\\
    & 
    c_{t}^{\mydot}(\boldsymbol{\xi}^{t}) = C_{t}^{\mydot}S_{t}\boldsymbol{\xi}, \;\;  \overline{e}_{t}(\boldsymbol{\xi}^{t}) = \overline{E}_{t}S_{t}\boldsymbol{\xi},
    \end{aligned}
\end{align}
for investment and fuel cost, demand and carbon cap stochastic processes, where matrices $Q_{t}^{\mydot},L_{t},C_{t}^{\mydot}\in\mathbb{R}^{N\times n^{t}}$ and $\overline{E}_{t}\in\mathbb{R}^{1\times n^{t}},$ $\forall t\in\llbracket T\rrbracket$ are devised by experts. Here, we use subscripts $\odot=\{\fl{g,s,p}\}$ for investment cost $Q_{t}^{\mydot}$, and $\odot=\{\fl{e,c}\}$ for fuel cost $C_{t}^{\mydot}$. These matrices are constructed such that their first columns collect the $1^{\text{st}}$-stage value of the parameter of interest, and the remaining columns include additive changes of that parameter across the investment horizon. If there is no uncertainty, i.e., $\boldsymbol{\xi}=\mathbb{E}[\boldsymbol{\xi}]=\mathbb{1}$, then the multiplication of these matrices by a unit vector yields a deterministic value of that parameter. The distribution of random variable $\boldsymbol{\xi}$ thus encodes the uncertainty of additive changes of planning parameters. The code in the e-companion in \cite{dvorkin2022} details the construction of matrices in \eqref{eq:uncertain_data_def}. To account for the errors in uncertainty estimation, we suppose that the exact distribution of the random variable is unknown, but it belongs to the following family of distributions:
\begin{align}\label{eq:amb_set}
\mathcal{P} = \{\mathbb{P}:\mathbb{E}[\boldsymbol{\xi}]=\mathbb{1},\;\mathbb{E}[\boldsymbol{\xi}\boldsymbol{\xi}^{\top}]=\Sigma\},
\end{align}
where set $\mathcal{P}$ encapsulates all probability measures with given first- and second-order moments $\mathbb{1}$ and $\Sigma$, respectively. 

\subsection{Chance-Constrained Problem Formulation}
To optimize investment decisions under planning uncertainty, we put forth the multi-stage chance-constrained program \eqref{prog:sto}, which minimizes the expected cost, i.e., 
\begin{subequations}\label{prog:sto}
\begin{align}
&\mathbb{E}\Bigg[\sum_{t=1}^{T}\bigg(
\underbrace{q_{t}^{\text{g}}(\boldsymbol{\xi}^{t})^{\top}\overline{y}_{t}(\boldsymbol{\xi}^{t}) + 
q_{t}^{\text{s}}(\boldsymbol{\xi}^{t})^{\top}\overline{\vartheta}_{t}(\boldsymbol{\xi}^{t}) + 
q_{t}^{\text{p}}(\boldsymbol{\xi}^{t})^{\top}\overline{\varphi}_{t}(\boldsymbol{\xi}^{t})}_{\textcolor{maincolor}{\text{stochastic annualized investment cost}}}  \nonumber\\
&+\underbrace{o_{t}^{\text{e}\top}\overline{p}_{t} +\sum_{\tau=1}^{t}\Big(
o_{t}^{\text{c}\top}\overline{y}_{\tau}(\boldsymbol{\xi}^{\tau}) + 
o_{t}^{\text{s}\top}\overline{\vartheta}_{\tau}(\boldsymbol{\xi}^{\tau}) + 
o_{t}^{\text{p}\top}\overline{\varphi}_{\tau}(\boldsymbol{\xi}^{\tau})\Big)}_{\textcolor{maincolor}{\text{stochastic annualized operation and maintenance (O\&M) cost}}}  \nonumber \\
&+\underbrace{\!\sum_{w=1}^{W}\!\omega_{w}\!\sum_{h=1}^{H}\big(
c_{t}^{\text{e}}(\boldsymbol{\xi}^{t})^{\top}p_{twh}(\boldsymbol{\xi}^{t}) + 
c_{t}^{\text{c}}(\boldsymbol{\xi}^{t})^{\top}y_{twh}(\boldsymbol{\xi}^{t})\big)\!\!\bigg)}_{\textcolor{maincolor}{\text{stochastic annualized fuel cost}}}\!\Bigg],
\label{prog:sto_obj}
\end{align}
where expectation $\mathbb{E}$ is taken for the worst-case uncertainty distribution in ambiguity set $\mathcal{P}$, subject to stochastic equalities and \textcolor{maincolor}{probabilistic} constraints, i.e., 
\begin{align}
&\left.
\begin{aligned}
&\mathbb{1}^{\top}\big(p_{twh}(\boldsymbol{\xi}^{t}) + y_{twh}(\boldsymbol{\xi}^{t}) + \varphi_{twh}^{\shortminus}(\boldsymbol{\xi}^{t})\\
&\quad- k_{twh}^{\ell}\circ\ell_{t}(\boldsymbol{\xi}^{t}) - \varphi_{twh}^{\shortplus}(\boldsymbol{\xi}^{t}) \big)= 0\\
&\vartheta_{twh}(\boldsymbol{\xi}^{t})=\vartheta_{tw(h-1)}(\boldsymbol{\xi}^{t})\\
&\quad+\varphi_{twh}^{\shortplus}(\boldsymbol{\xi}^{t})\eta^{\shortplus}-\varphi_{twh}^{\shortminus}(\boldsymbol{\xi}^{t})/\eta^{\shortminus}
\end{aligned}
\right\}\forall\boldsymbol{\xi}^{t},\label{cc_as}\\
&\mathbb{P}\!\!\left[
    \begin{aligned}
    &|F\big(p_{twh}(\boldsymbol{\xi}^{t}) + y_{twh}(\boldsymbol{\xi}^{t}) + \varphi_{twh}^{\shortminus}(\boldsymbol{\xi}^{t})\\
    &\quad - k_{twh}^{\ell}\circ\ell_{t}(\boldsymbol{\xi}^{t}) - \varphi_{twh}^{\shortplus}(\boldsymbol{\xi}^{t}) \big)|\leqslant\overline{f}
    \end{aligned}
\right]\!\!\geqslant\!1\!-\!\varepsilon^{\fl{f}},\label{cc_flow}\\
&\mathbb{P}\!\!\left[
    \begin{aligned}
    &\mathbb{0}\leqslant p_{twh}(\boldsymbol{\xi}^{t}) \leqslant k_{twh}^{\fl{e}}\circ\overline{p}_{t}\\
    &\mathbb{0}\leqslant y_{twh}(\boldsymbol{\xi}^{t}) \leqslant k_{twh}^{\fl{c}}\circ \textstyle\sum_{\tau=1}^{t}\overline{y}_{\tau}(\boldsymbol{\xi}^{\tau})
    \end{aligned}
\right]\!\!\geqslant\!1\!-\!\varepsilon^{\fl{g}},\label{cc_gen}\\
&\mathbb{P}\!\!\left[
    \begin{aligned}
    &-r^{\fl{e}\shortminus}\!\circ\overline{p}_{t} \leqslant p_{twh}(\boldsymbol{\xi}^{t}) - p_{tw(h-1)}(\boldsymbol{\xi}^{t}) \leqslant r^{\fl{e}\shortplus}\!\circ\overline{p}_{t}\\
    &y_{twh}(\boldsymbol{\xi}^{t}) - y_{tw(h-1)}(\boldsymbol{\xi}^{t}) \geqslant -r^{\fl{c}\shortminus}\circ\textstyle\sum_{\tau=1}^{t}\overline{y}_{\tau}(\boldsymbol{\xi}^{\tau})\\
    &y_{twh}(\boldsymbol{\xi}^{t}) - y_{tw(h-1)}(\boldsymbol{\xi}^{t})  \leqslant r^{\fl{c}\shortplus} \circ\textstyle\sum_{\tau=1}^{t}\overline{y}_{\tau}(\boldsymbol{\xi}^{\tau})
    \end{aligned}
\right]\!\!\geqslant\!1\!-\!\varepsilon^{\fl{r}},\!\!\label{cc_ramp}\\[-2.5ex]
&\mathbb{P}\!\!\left[
    \begin{aligned}
    &\mathbb{0}\leqslant \vartheta_{twh}(\boldsymbol{\xi}^{t}) \leqslant \textstyle\sum_{\tau=1}^{t}\overline{\vartheta}_{\tau}(\boldsymbol{\xi}^{\tau})\\
    &\mathbb{0}\leqslant\varphi_{twh}^{\shortplus}(\boldsymbol{\xi}^{t}) \leqslant \textstyle\sum_{\tau=1}^{t}\overline{\varphi}_{\tau}(\boldsymbol{\xi}^{\tau})\\
    &\mathbb{0}\leqslant\varphi_{twh}^{\shortminus}(\boldsymbol{\xi}^{t}) \leqslant \textstyle\sum_{\tau=1}^{t}\overline{\varphi}_{\tau}(\boldsymbol{\xi}^{\tau})\\
    &\varphi_{twh}^{\shortplus}(\boldsymbol{\xi}^{t}) + \varphi_{twh}^{\shortminus}(\boldsymbol{\xi}^{t}) \leqslant \textstyle\sum_{\tau=1}^{t}\overline{\varphi}_{\tau}(\boldsymbol{\xi}^{\tau})
    \end{aligned}
\right]\!\!\geqslant\!1\!-\!\varepsilon^{\fl{s}},\!\!\label{cc_stor}\\
&\mathbb{P}\!\!\left[
    \begin{aligned}
    &\textstyle\sum_{w=1}^{W}\omega_{w}\textstyle\sum_{h=1}^{H}\big(e^{\fl{e}\top}p_{twh}(\boldsymbol{\xi}^{t}) \\
    &\quad+ e^{\fl{c}\top}y_{twh}(\boldsymbol{\xi}^{t})\big) \leqslant \overline{e}_{t}(\boldsymbol{\xi}^{t})
    \end{aligned}
\right]\!\!\geqslant\!1\!-\!\varepsilon^{\fl{e}},\label{cc_emis}\\
&\mathbb{P}\!\!\left[
    \begin{aligned}
    &\mathbb{0}\leqslant\overline{y}_{t}(\boldsymbol{\xi}^{t})\leqslant \overline{y}_{t}^{\fl{max}}\\
    &\mathbb{0}\leqslant\overline{\varphi}_{t}(\boldsymbol{\xi}^{t})\leqslant \overline{\varphi}_{t}^{\fl{max}}\\
    &\mathbb{0}\leqslant\overline{\vartheta}_{t}(\boldsymbol{\xi}^{t})\leqslant \overline{\vartheta}_{t}^{\fl{max}}
    \end{aligned}
\right]\!\!\geqslant\!1\!-\!\varepsilon^{\fl{i}},\label{cc_inv}\\
&\textcolor{maincolor}{\mathbb{E}
\big[
q_{t}^{\text{g}}(\boldsymbol{\xi}^{t})^{\top}\overline{y}_{t}(\boldsymbol{\xi}^{t}) + 
q_{t}^{\text{s}}(\boldsymbol{\xi}^{t})^{\top}\overline{\vartheta}_{t}(\boldsymbol{\xi}^{t})} \nonumber\\
&\quad\quad\quad\quad\quad\quad\quad\quad+ 
\textcolor{maincolor}{q_{t}^{\text{p}}(\boldsymbol{\xi}^{t})^{\top}\overline{\varphi}_{t}(\boldsymbol{\xi}^{t})
\big]\leqslant b_{t},}\label{exp_budget}\\
&\forall\mathbb{P}\in\mathcal{P},\;\forall t\in\llbracket T\rrbracket,\;\forall w\in\llbracket W\rrbracket,\;\forall h\in\llbracket H\rrbracket,\nonumber
\end{align}
\end{subequations}
where decision variables are required to depend solely on the current- and previous-stage uncertainty $\boldsymbol{\xi}^{t}$, hence reflecting the non-anticipative nature of the investment problem. Stochastic equalities \eqref{cc_as} must hold for all uncertainty realizations from the unknown distribution $\mathbb{P}_{\boldsymbol{\xi}^{t}}$, which is reasonable as the investment plan that violates energy flow and storage conservation laws is of no interest to planning. The series of chance constraints \eqref{cc_flow}--\eqref{cc_inv} respectively require the joint satisfaction of power flow, generation dispatch, ramping, storage dispatch, carbon policy and investment constraints to be satisfied with probabilities at least $1-\varepsilon^{\mydot}\geqslant0$, where $\varepsilon^{\mydot}$, for $\odot=\{\fl{f,g,r,s,e,i}\}$, are prescribed parameters. Parameters $\varepsilon^{\mydot}$ internalize the tolerance to the joint constraint violation of the respective group of constraints. For example, Bienstock \textit{et al.} in \cite{bienstock2014chance} argue that assigning $\varepsilon^{\fl{f}}\gg\varepsilon^{\fl{g}}$ is reasonable, since the short-term overloading of overhead transmission lines can be tolerated in practice, while generation constraints must be respected with a very high probability. Below, we will show that probabilities $\varepsilon^{\mydot}$ act as trade-off parameters between the cost \eqref{prog:sto_obj} of the investment plan and its feasibility. \textcolor{maincolor}{Finally, constraint \eqref{exp_budget} requires the expected investment cost to remain below a prescribed threshold $b_{t}$ at each investment stage.
} 

\section{Stochastic Problem Reformulation in LDRs}\label{sec:cc_ldr}
To obtain a tractable approximation for chance-constrained problem \eqref{prog:sto}, the form of the decision variables is restricted to linear functions. That is, the investment decisions are now representable as linear decision rules
\begin{subequations}\label{eq:LDRs}
\begin{align}
\overline{y}_{t}(\boldsymbol{\xi}^{t}) = \overline{Y}_{t}S_{t}\boldsymbol{\xi}, 
\;\; 
\overline{\vartheta}_{t}(\boldsymbol{\xi}^{t}) = \overline{\Theta}_{t} S_{t}\boldsymbol{\xi},
\;\; 
\overline{\varphi}_{t}(\boldsymbol{\xi}^{t}) = \overline{\Phi}_{t} S_{t}\boldsymbol{\xi},
\end{align}
where $\overline{Y}_{t},\overline{\Theta}_{t},\overline{\Phi}_{t}\in\mathbb{R}^{N\times n^{t}}$ are matrices that define investment response to uncertainty $\forall t\in\llbracket T\rrbracket$, thus subject to optimization. Operational variables are also expressed as linear functions
\begin{align}
\begin{aligned}
y_{twh}(\boldsymbol{\xi}^{t}) = Y_{twh} S_{t}\boldsymbol{\xi}, 
\quad 
\vartheta_{twh}(\boldsymbol{\xi}^{t}) = \Theta_{twh}S_{t}\boldsymbol{\xi},
\\
\varphi_{twh}^{\mydot}(\boldsymbol{\xi}^{t}) = \Phi_{twh}^{\mydot}S_{t}\boldsymbol{\xi},
\quad 
p_{twh}(\boldsymbol{\xi}^{t}) = P_{twh}S_{t}\boldsymbol{\xi},
\end{aligned}
\end{align}
\end{subequations}
that are subject to optimization. 
\textcolor{maincolor}{To capture the LDR intuition, consider the expansion of the generation investment LDR:
\begin{align}\label{eq:LDR_expansion}
\overline{y}_{t}(\boldsymbol{\xi}^{t}) = \overline{Y}_{t}S_{t}\boldsymbol{\xi} = 
\begin{bmatrix}
\overline{Y}_{11} & \overline{Y}_{12} & \dots & \overline{Y}_{1t}
\\
\vdots & \vdots & \ddots & \vdots\\
\overline{Y}_{N1} & \overline{Y}_{N2} & \dots & \overline{Y}_{Nt}
\end{bmatrix}
\begin{bmatrix}
1\\
\boldsymbol{\xi}_{2}\\
\vdots\\
\boldsymbol{\xi}_{t}
\end{bmatrix},
\end{align}
where the rows of $\overline{Y}_{t}$ identify the response of the investments in $N$ candidate generation units at stage $t$ to the realizations of random vector $\boldsymbol{\xi}^{t}$. By design, the first column of $\overline{Y}_{t}$ collects the nominal investment decisions, which are then adjusted by recourse decisions in the remaining columns. A row of matrix $\overline{Y}_{t}$ computes the response of investments into a particular generation unit to random variables realized up to stage $t$. For example, suppose that the planner faces the uncertainty of demand $(\xi^{d})$, investment costs $(\xi^{q})$, and operating costs $(\xi^{c})$ starting from stage 2 onward. Then, for a particular generation unit $i\in\llbracket N\rrbracket$, the investment rule $\overline{y}_{it}(\boldsymbol{\xi}^{t}) $ expands as: 
\begin{align*}
    \overline{Y}_{i1} + \sum_{j=2}^{t}\overline{Y}_{ij}\boldsymbol{\xi}_{j}= 
    \overline{Y}_{i1} + \sum_{j=2}^{t}\left(\overline{Y}_{ij1}\xi_{j}^{d}+\overline{Y}_{ij2}\xi_{j}^{q}+\overline{Y}_{ij3}\xi_{j}^{c}\right),
\end{align*}
where matrix $\overline{Y}_{ij}\in\mathbb{R}^{1\times3}$ computes the investment response to the three uncertainty sources at stage $j\geqslant2$. This way, the optimized matrix $\overline{Y}_{t}$ guides the entire generation investment portfolio as uncertainties gradually realize throughout the investment horizon. While this linear functional form allows for a tractable reformulation of problem \eqref{prog:sto}, the linear response can be sub-optimal. In Section \ref{sec:sub-optimality_gurantees} we develop two methods to assess the sub-optimality of LDRs, and we report numerical results in Section \ref{sec:case_study}. 
\begin{remark}
To model discrete investments (e.g., for indivisible or modular generation and energy storage units) the linear investment decision rules in \eqref{eq:LDRs} can be substituted with their binary counterparts \cite{bertsimas2018binary}. This, in turn, will require the use of lifting techniques from \cite{zhang2018ambiguous} to preserve the guarantees of chance constraints enforced on binary decisions. Another approach is to solve a projection problem that projects the optimized linear decision rule under a specific scenario onto a discrete space of available investment options. The authors relegate the integrality of investment decisions to future work. 
\end{remark}}

The uncertainty representation \eqref{eq:uncertain_data_def} and LDRs \eqref{eq:LDRs} allow for the following scenario-free reformulation of problem \eqref{prog:sto}. First, the objective function reformulates considering that $\mathbb{E}[S_{t}\boldsymbol{\xi}] = S_{t}\mathbb{1}$ and $\mathbb{E}[S_{t}\boldsymbol{\xi}(S_{t}\boldsymbol{\xi})^\top]=\text{Tr}[S_{t}(\Sigma + \mathbb{1}\mathbb{1}^{\top}) S_{t}^{\top}]$, where $\mathbb{1}$ and $\Sigma$ are the moments of $\boldsymbol{\xi}$. We set $\widehat{\Sigma}=\Sigma + \mathbb{1}\mathbb{1}^{\top}$ for brevity.  Then, objective function \eqref{prog:sto_obj} reformulates analytically into
\begin{subequations}\label{prog:LDR_approximation}
\begin{align}
&\sum_{t=1}^{T}\Big(
\text{Tr}[S_{t}\widehat{\Sigma} S_{t} (Q_{t}^{\text{g}\top}\overline{Y}_{t} + 
Q_{t}^{\text{s}\top}\overline{\Theta}_{t} +
Q_{t}^{\text{p}\top}\overline{\Phi}_{t})
]\nonumber\\
&\quad+o_{t}^{\text{e}\top}\overline{P}_{t}S_{t}\mathbb{1} 
+\sum_{\tau=1}^{t}\Big(
o_{t}^{\text{c}\top}\overline{Y}_{\tau}  + {}
o_{t}^{\text{s}\top}\overline{\Theta}_{\tau} + 
o_{t}^{\text{p}\top}\overline{\Phi}_{\tau}\Big)S_{\tau}\mathbb{1}  \nonumber \\
&\quad\quad+\sum_{w=1}^{W}\omega_{w}\sum_{h=1}^{H}\text{Tr}[
S_{t}\widehat{\Sigma} S_{t} (
C_{t}^\fl{e}P_{twh} + C_{t}^\fl{c}Y_{twh})
]\Big).\label{prog:LDR_obj}
\end{align}
Second, to reformulate stochastic equalities \eqref{cc_as}, we use the fact that enforcing stochastic equation $X_{t}S_{t}\boldsymbol{\xi}=\mathbb{0}$, for some properly dimensioned matrix $X_{t}$, is equivalent to enforcing a deterministic equation $X_{t}=\mathbb{0}$. Hence, the power balance will hold for any uncertainty realization when 
\begin{align}
    \mathbb{1}^{\top}\big(P_{twh} + Y_{twh} + \Phi_{twh}^{\shortplus} - k_{twh}^{\ell}\circ L_{t} - \Phi_{twh}^{\shortminus} \big) = \mathbb{0},\label{LDR_eq_balance}
\end{align}
and the storage state of charge equation will be satisfied when 
\begin{align}
    \Theta_{twh} - \Theta_{tw(h-1)} - \Phi_{twh}^{\shortplus}\eta^{\shortplus}+ \Phi_{twh}^{\shortminus}/\eta^{\shortminus} = \mathbb{0},\label{LDR_eq_soc}
\end{align}
at all time stages, operating horizons and representative hours. Observe, that the objective function and equality reformulations in \eqref{prog:LDR_obj}--\eqref{LDR_eq_soc} do not require distributional assumptions.

To reformulate individual and joint chance constraints in a distributionally robust manner, we use Chebyshev inequality and Bonferroni approximation \cite{xie2017distributionally,xie2019optimized}. \textcolor{maincolor}{From \cite{xie2017distributionally} we know that an individual, distributionally robust, and single-sided chance constraint of the form
\begin{align*}
    \mathbb{P}[\boldsymbol{\xi}^{\top}x\leqslant b]\geqslant1-\varepsilon
\end{align*}
 translates into tractable second-order cone constraint
\begin{align*}
    \sqrt{(1-\varepsilon)/\varepsilon}\norm{\overline{\Sigma}x}\leqslant b-x^{\top}\mathbb{1},
\end{align*}
where $\overline{\Sigma}$ is the Cholesky decomposition of the covariance matrix, i.e., $\ncoverline{\Sigma}\ncoverline{\Sigma}^{\top}=\Sigma$, and the square root term is the safety factor from the Chebyshev inequality to ensure distributional robustness. The smaller the $\varepsilon$, the large the left-hand-side of the second-order cone constraint, hence the most robust the nominal solution $x$ to uncertainty realizations. Following the same reformulation, individual carbon constraints in \eqref{cc_emis} reformulate into the following second-order cone constraints:}
\begin{align}
&\tilde{\varepsilon}^{\fl{e}}\norm{\overline{\Sigma}
\left[
\Big(\overline{E}_{t}\!-\!\textstyle\sum_{w=1}^{W}\omega_{w}\textstyle\sum_{h=1}^{H}\big(e^{\fl{e}\top}P_{twh}\!+\!e^{\fl{c}\top}Y_{twh}\big)\Big)S_{t}
\right]^{\top}
}\nonumber\\
&\leqslant\Big(\overline{E}_{t}\!-\!\textstyle\sum_{w=1}^{W}\omega_{w}\textstyle\sum_{h=1}^{H}\big(e^{\fl{e}\top}P_{twh}\!+\!e^{\fl{c}\top}Y_{twh}\big)\Big)S_{t}\mathbb{1},\label{eq:single_sided_ref}
\end{align}
for all investment stages $t$, where prescribed parameter 
\textcolor{maincolor}{$\tilde{\varepsilon}^{\fl{e}}=\sqrt{(1-\varepsilon^{\fl{e}})/\varepsilon^{\fl{e}}}$ to ensure distributional robustness.} 

While such a reformulation fits the single-sided constraints, it is generally overly conservative for double-sided constraints, because for any uncertainty realization, such constraints can not be violated simultaneously from both sides. Indeed, decision variables for power flows in \eqref{cc_flow}, existing generation in \eqref{cc_gen}, and investments in \eqref{cc_inv} can not violate their respective minimum and maximum limits simultaneously. To address this issue, we first split the joint chance constraints in a series of individual double-sided constraints (as per Bonferroni approximation \cite{xie2019optimized}), and then apply \cite[Theorem 2]{xie2017distributionally} to obtain its less conservative reformulation. \textcolor{maincolor}{A joint chance constraint 
\begin{align*}
    \mathbb{P}[\underline{y}\leqslant Y\boldsymbol{\xi}\leqslant \overline{y}]\geqslant1-\varepsilon
\end{align*}
with variable $Y\in\mathbb{R}^{k\times n}$ and parameters $\underline{y},\overline{y}\in\mathbb{R}^{k}$ is split into $k$ individual chance constraints as 
\begin{align*}
    \mathbb{P}[\underline{y}_{i}\leqslant Y_{i}^{\top}\boldsymbol{\xi}\leqslant \overline{y}_{i}]\geqslant1-\overline{\varepsilon}_{i},\quad\forall i\in\llbracket k\rrbracket,
\end{align*}
where were require that $\sum_{i=1}^{k}\overline{\varepsilon}_{i}=\varepsilon$ to preserve the joint constraint satisfaction guarantee. Then, the double-sided individual chance constraints reformulate into a set of one second-order cone and five linear constraints as: 
\begin{align*}\left\{
    \begin{aligned}
    &\norm{
    \begin{bmatrix*}
    \overline{\Sigma}Y_{i}^{\top}\\
    z_{i}
    \end{bmatrix*}
    }\leqslant
    \sqrt{\overline{\varepsilon}_{i}}\left(\frac{\overline{y}_{i} - \underline{y}_{i}}{2} - x_{i}\right)\\
    &\left|Y_{i}^{\top}\mathbb{1}\right|\leqslant z_{i}+x_{i}, \quad \frac{\overline{y}_{i} - \underline{y}_{i}}{2} \geqslant x_{i} \geqslant 0, \quad z_{i} \geqslant 0
    \end{aligned}\right.
\end{align*}
where $z_{i}$ and $x_{i}$ are two auxiliary variables. Following the same method, the joint constraint \eqref{cc_inv} is split into $3N$ individual double-sided constraints, i.e., for the $i^{\fl{th}}$ candidate generation unit we have
}
\begin{align}
\mathbb{P}\!\!\left[
    0\leqslant\overline{y}_{ti}(\boldsymbol{\xi}^{t})\leqslant \overline{y}_{ti}^{\fl{max}}
\right]\geqslant 1 - \overline{\varepsilon}^{\fl{i}},
\end{align}
where $\overline{\varepsilon}^{\fl{i}} = \varepsilon^{\fl{i}}/(3N)$. Then its exact distributionally robust reformulation takes the form: 
\begin{align}{}
&\norm{
\begin{bmatrix*}
\overline{\Sigma}[\overline{Y}_{t}S_{t}]_{i}^{\top}\\
z_{ti}^{\overline{y}}
\end{bmatrix*}
}\leqslant
\sqrt{\overline{\varepsilon}^{\fl{i}}}\left(\tfrac{1}{2}\overline{y}_{ti}^{\fl{max}} - x_{ti}^{\overline{y}}\right),\label{cc_inv_ref_1}\\
&\left|[\overline{Y}_{t}S_{t}]_{i}\mathbb{1} - \tfrac{1}{2}\overline{y}_{ti}^{\fl{max}}\right| \leqslant z_{ti}^{\overline{y}} + x_{ti}^{\overline{y}}\\
&\tfrac{1}{2}\overline{y}_{ti}^{\fl{max}} \geqslant x_{ti}^{\overline{y}} \geqslant 0,\quad z_{ti}^{\overline{y}} \geqslant 0,\label{cc_inv_ref_3}
\end{align}
\end{subequations}
where $x_{ti}^{\overline{y}}$ and $z_{ti}^{\overline{y}}$ are auxiliary variables. The rest of double-sided constraints are reformulated in a similar manner. 

\textcolor{maincolor}{Finally, the left-hand side of the investment budget constraint \eqref{exp_budget} is reformulated in the same way as the investment cost term in the objective function \eqref{prog:LDR_obj}.}

We refer to Appendix \ref{app:full_tract_ref} for the full reformulation of problem \eqref{prog:sto} into a second-order cone program. 

\section{Variance-Constrained Investments: From Stochastic to Quasi-Deterministic Planning} \label{sec:inv_var_red_and_det_eq}

The multi-stage stochastic problem \eqref{prog:sto} optimizes the expected cost of investment planning, which is ignorant of the variance of investment results. Moreover, it has been shown in prior work that restricting recourse decisions to LDRs makes optimization decisions highly sensitive to uncertainty \cite{dvorkin2021stochastic}. Such risk-neutral LDR optimization tends to produce an unhedged variance of investment decisions, thus making decarbonization pathways highly uncertain in terms of investment costs and structure. In this section, we develop a variance-constrained investment optimization, which acts on the same uncertainty information as problem \eqref{prog:sto} but outputs such an investment plan, which is less sensitive (or even insensitive) to planning uncertainty, yet remains robust to its realizations. We term such a plan \textit{quasi-deterministic}. Such a plan can eventually identify the optimal adjustment of the deterministic investment plan from problem \eqref{prog:base} to immunize investments and future system operations against uncertainty.

To obtain the quasi-deterministic investment plan, observe that the investment variance is a convex function in investment LDRs. For example, for generation investments, we have\begin{align}\label{eq:variance}
\text{Var}[\overline{Y}\!_{t}S_{t}\boldsymbol{\xi}]\!=\!\mathbb{E}[(\overline{Y}\!_{t}S_{t}\boldsymbol{\xi})(\overline{Y}\!_{t}S_{t}\boldsymbol{\xi})\!^{\top}] 
\!=\! 
\text{Tr}[\overline{Y}\!_{t}S_{t}\widehat{\Sigma} (\overline{Y}\!_{t}S_{t})\!^{\top}],
\end{align}
which is a convex function in variable $\overline{Y}_{t}$. Hence, to produce the variance-constrained solution, it is sufficient to optimize coefficient matrix $\overline{Y}_{t}$ while constraining the magnitude of \eqref{eq:variance}. Towards the goal, we use the following proxy constraints:
\begin{subequations}\label{prog:var_con}
\begin{align}
&\norm{\overline{\Sigma}[\overline{Y}_{t}S_{t}]_{i}^{\top}}\leqslant\alpha_{i}^{\overline{y}}[\overline{Y}_{t}S_{t}]_{i}\mathbb{1},\quad \forall i\in\llbracket N\rrbracket, \quad \forall t\in\llbracket T\rrbracket,\label{var_con_1}\\
&\norm{\overline{\Sigma}[\overline{\Theta}_{t}S_{t}]_{i}^{\top}}\leqslant\alpha_{i}^{\overline{\vartheta}}[\overline{\Theta}_{t}S_{t}]_{i}\mathbb{1},\quad \forall i\in\llbracket N\rrbracket, \quad \forall t\in\llbracket T\rrbracket,\\
&\norm{\overline{\Sigma}[\overline{\Phi}_{t}S_{t}]_{i}^{\top}}\leqslant\alpha_{i}^{\overline{\varphi}}[\overline{\Phi}_{t}S_{t}]_{i}\mathbb{1},\quad \forall i\in\llbracket N\rrbracket, \quad \forall t\in\llbracket T\rrbracket,\label{var_con_3}
\end{align}
\end{subequations}
where the norm terms compute the standard deviation of the respective investment decisions into generation and storage capacities, and the right-hand-side terms compute the mean values of those decisions multiplied by the prescribed non-negative coefficients $\alpha^{\mydot}$. Thus, the standard deviation -- and the variance -- of the investment plan can now be controlled by a certain portion of the mean value when adding constraints \eqref{prog:var_con} to the base LDR optimization problem \eqref{prog:full_tract_ref}. Notice that setting $\alpha^{\mydot} \rightarrow+\infty$ results in the variance-agnostic solution and $\alpha^{\mydot} \rightarrow0$ yields a zero-variance investment solution. The latter is always feasible because the norms in \eqref{prog:var_con} do not restrict the first column of LDRs. For example, for $T=n=2$, the norm in the generation LDR at the second stage expands as
\begin{align*}
    \norm{\!
    \begin{bmatrix}
    0 & 0\\
    0 & \sigma_{2}
    \end{bmatrix}
    \!
    \begin{bmatrix}
    \overline{Y}_{11} & \overline{Y}_{12} \\
    \vdots & \vdots \\
    \overline{Y}_{N1} & \overline{Y}_{N2}
    \end{bmatrix}_{i}^{\top}
    \!}\!=\!
    \norm{\!
    \begin{bmatrix}
    0 & 0\\
    0 & \sigma_{2}
    \end{bmatrix}
    \!
    \begin{bmatrix}
    \overline{Y}_{i1} \\ \overline{Y}_{i2}
    \end{bmatrix}
    \!}\!=\!
    \norm{\!
    \begin{bmatrix}
    \sigma_{2}\overline{Y}_{i2}
    \end{bmatrix}
    \!},
\end{align*}
meaning that there always exists a nominal generation investment decision $\overline{Y}_{i1}$ which makes the variance-constrained solution feasible. However, imposing constraints \eqref{prog:var_con} will provide a more conservative solution in terms of expected costs. By varying $\alpha^{\mydot}$, we will eventually reveal the cost-variance trade-off in the investment planning under uncertainty.

\section{On Investment Decision Rule Sub-Optimality} \label{sec:sub-optimality_gurantees}

The linear variable dependency on the planning uncertainty in problem \eqref{prog:sto} is an additional, implicit constraint on investment planning. Hence, the LDR-based planning is likely to demonstrate additional economic inefficiency compared to the less conservative -- yet poorly scalable -- scenario-based stochastic programming. Therefore, this section provides the framework for the quantitative and qualitative assessment of the investment LDR sub-optimality to offer the corresponding \textit{a priori} performance guarantees. 

\subsection{Global LDR Sub-Optimality Bound}

We first develop the global bound to guarantee that the preference of LDRs over scenario-based stochastic programming does not lead to financial losses exceeding this bound. Our bound is inspired by the duality-based method from robust optimization in \cite{kuhn2011primal}, which we extend here to the case of chance-constrained optimization. To compute this bound, consider first a stylized, compact version of problem \eqref{prog:sto}:
\begin{subequations}\label{prog:bound_inf_primal}
\begin{align}
\minimize{x\in\mathbb{R}^{n}}\quad& \text{P}(x) \triangleq \mathbb{E}\left[\textstyle\sum_{t=1}^{T}c_{t}(\boldsymbol{\xi}^{t})^{\top}x_{t}(\boldsymbol{\xi}^{t})\right]\\
\st\quad&\mathbb{P}\left[\textstyle\sum_{\tau=1}^{t}A_{\tau}x_{\tau}(\boldsymbol{\xi}^{\tau})\geqslant b_{t}(\boldsymbol{\xi}^{t})\right]\geqslant1-\varepsilon_{t},
\end{align}
\end{subequations}
$\forall t\in\llbracket T\rrbracket$, where the expected planning cost is minimized subject to the joint chance constraints and where vectors $c=(c_{1},\dots,c_{T})$ and $b=(b_{1},\dots,b_{T})$ are uncertain and modeled similarly to \eqref{eq:uncertain_data_def}. Notation $\text{P}(x)$ denotes the primal optimization objective, which is a function of the investment and operational decisions collected in vector $x=(x_{1},\dots,x_{T})$. For primal problem \eqref{prog:bound_inf_primal}, we formulate the corresponding dual problem: 
\begin{subequations}\label{prog:bound_inf_dual}
\begin{align}
\maximize{\lambda\in\mathbb{R}^{m}}\quad&\text{D}(\lambda) \triangleq \mathbb{E}\left[\textstyle\sum_{t=1}^{T}b_{t}(\boldsymbol{\xi}^{t})^{\top}\lambda_{t}(\boldsymbol{\xi}^{t})\right]\\
\st\quad&\mathbb{P}\left[\textstyle\sum_{\tau=t}^{T}A_{t}^{\top}\lambda_{\tau}(\boldsymbol{\xi}^{\tau})\leqslant c_{t}(\boldsymbol{\xi}^{t})\right]\geqslant1-\varepsilon_{t},
\end{align}
\end{subequations}
$\forall t\in\llbracket T\rrbracket,$ where the dual objective function $\text{D}(\lambda)$ in variable $\lambda=(\lambda_{1},\dots,\lambda_{T})$ is maximized subject to the dual joint chance constraints. Problem \eqref{prog:bound_inf_dual} is a stylized, compact version of the dual problem of \eqref{prog:sto}, which we place in Appendix \ref{app:sto_dual} for completeness. Recall, that solving the two problems to optimality using scenario-based stochastic programming is computationally intractable at scale. However, we can approximate the optimal solution $(x^{\star},\lambda^\star)$ using linear decision rules $x_{t}(\boldsymbol{\xi}^{t}) = X_{t}S_{t}\boldsymbol{\xi}$ and $\lambda_{t}(\boldsymbol{\xi}^{t}) = \Lambda_{t}S_{t}\boldsymbol{\xi}$ for all $t\in\llbracket T\rrbracket$, and by solving the primal LDR approximation problem
\begin{subequations}\label{prog:bound_apprx_primal}
\begin{align}
\minimize{X}\quad&\textstyle \overline{\text{P}}(X) \triangleq \sum_{t=1}^{T}\text{Tr}\big[S_{t}\widehat{\Sigma}S_{t}^{\top}C_{t}^{\top}X_{t}\big]\\
\st\quad&\tilde{\varepsilon}_{t}\norm{\overline{\Sigma}[\textstyle\sum_{\tau=1}^{t}A_{\tau}X_{\tau}S_{\tau}-B_{t}S_{t}]_{i}^\top}\nonumber\\
&\quad\quad\quad\leqslant[\textstyle\sum_{\tau=1}^{t}A_{\tau}X_{\tau}S_{\tau}-B_{t}S_{t}]_{i}\mathbb{1},\nonumber\\
&\quad\quad\quad\quad\quad\quad\quad\quad\quad\forall i\in\llbracket m\rrbracket,\;\forall t\in\llbracket T\rrbracket,
\end{align}{}
\end{subequations}
and the dual LDR approximation problem
\begin{subequations}\label{prog:bound_apprx_dual}
\begin{align}
\maximize{\Lambda}\quad&\textstyle \overline{\text{D}}(\Lambda) \triangleq \sum_{t=1}^{T}\text{Tr}\big[S_{t}\widehat{\Sigma}S_{t}^{\top}B_{t}^{\top}\Lambda_{t}\big]\\
\st\quad&\tilde{\varepsilon}_{t}\norm{\overline{\Sigma}[C_{t}S_{t}-\textstyle\sum_{\tau=t}^{T}A_{\tau}^\top\Lambda_{\tau}S_{\tau}]_{i}^\top}\nonumber\\
&\quad\quad\quad\leqslant[C_{t}S_{t}-\textstyle\sum_{\tau=t}^{T}A_{\tau}^\top\Lambda_{\tau}S_{\tau}]_{i}\mathbb{1},\nonumber\\
&\quad\quad\quad\quad\quad\quad\quad\quad\quad\forall i\in\llbracket n\rrbracket,\;\forall t\in\llbracket T\rrbracket,
\end{align}
\end{subequations}
which are obtained using the same means as those explained in Section \ref{sec:cc_ldr}. By solving these two problems, we can obtain the global bound on the sub-optimality of the investment LDRs. 
\begin{theorem}\label{th:gap}\normalfont
Let $X^{\star}$ and $\Lambda^{\star}$ be the optimal solution of the primal and dual approximation problems \eqref{prog:bound_apprx_primal} and \eqref{prog:bound_apprx_dual}, respectively. Then, the duality gap 
$\overline{\text{P}}(X^{\star})-\overline{\text{D}}(\Lambda^{\star})$ is the global bound on the LDR sub-optimality. 
\end{theorem}
\begin{proof}
Let $\tilde{x}$ and $\tilde{\lambda}$ be the optimal primal and dual solutions of stochastic problems \eqref{prog:bound_inf_primal} and \eqref{prog:bound_inf_dual}, respectively. Since the primal LDR approximation constrains the decisions to be linear in planning uncertainty, we know that the LDR objective function is equal or above the optimal solution, i.e., $\overline{\text{P}}(X^{\star}) \geqslant \text{P}(\tilde{x})$. Similarly, we know that the restrictive dual LDR approximation is equal or below the optimal solution, i.e.,  $\overline{\text{D}}(\Lambda^{\star}) \leqslant \text{D}(\tilde{\lambda})$. For problems \eqref{prog:bound_inf_primal} and \eqref{prog:bound_inf_dual}, at least the weak duality holds, i.e., $\text{P}(\tilde{x})\geqslant\text{D}(\tilde{\lambda}).$ Hence, the LDR duality gap is always non-negative, i.e., $\overline{\text{P}}(X^{\star}) - \overline{\text{D}}(\Lambda^{\star})\geqslant0,$ and globally bounds the LDR sub-optimality.  
\end{proof}

Using this result, we can \textit{a priori} guarantee that the average economic loss resulting from the implementation of the optimized investment LDRs does not exceed the global bound. 

\subsection{Learning the Worst-Case Sub-Optimality Scenarios}

This section extends the LDR sub-optimality analysis to learn the likelihood of the worst-case sub-optimality scenarios. This analysis is meant for the qualitative analysis and is limited to 1) small problem instances, for which the scenario-based solution can be retrieved, and to 2) investment problems where uncertainty solely enters the feasible region, e.g., demand and carbon target uncertainty. Consider, for example, a two-stage ($T=2$) instance of problem \eqref{prog:bound_inf_primal} and discretize the random variable $\boldsymbol{\xi}$ using $\Omega$ number of scenarios, i.e., $\xi_{1},\dots,\xi_{\Omega}$. The scenario-based version of \eqref{prog:bound_inf_primal} with the sole uncertainty in the feasible region takes the form:
\begin{subequations}\label{prog:saa_small}
\begin{align}
\minimize{x_{1},x_{2}}\quad& \tfrac{1}{\Omega}\textstyle\sum_{\omega=1}^{\Omega}(c_{1}^{\top}x_{1} + c_{2}^{\top}x_{2\omega})\label{saa_small_obj}\\
\st\quad&A_{1}x_{1}\geqslant B_{1}S_{1}\xi_{\omega},\label{saa_small_con_1}\\
&A_{1}x_{1}+A_{2}x_{2\omega}\geqslant B_{2}S_{2}\xi_{\omega},\;\forall \omega\in\llbracket \Omega\rrbracket,\label{saa_small_con_2}
\end{align}
\end{subequations}
where the objective function computes the sample average investment planning cost, and constraints \eqref{saa_small_con_1} and \eqref{saa_small_con_2} are written for the first and second stages, respectively. Observe, that the first-stage decision $x_{1}$ is uncertainty independent, while the second-stage decisions $x_{21},\dots,x_{2\Omega}$ are scenario-specific. To identify the worst-case sub-optimality gap between the solution of LDR approximation \eqref{prog:bound_apprx_primal} and its scenario-based counterpart \eqref{prog:saa_small}, we put forth the following bilevel program:
\begin{align}
    &\begin{array}{@{}rl@{}}
    \maximize{\hat{\xi}}&
    (c_{1}^{\top}X_{1}^{\star}\hat{\xi} + c_{2}^{\top}X_{2}^{\star}\hat{\xi}) - (c_{1}^{\top}x_{1}^{\star} + c_{2}^{\top}x_{2})
    \\
    \st& \underline{\xi}\leqslant\hat{\xi}\leqslant\overline{\xi},\\
    \end{array}\tag{UL}\label{UL}\\
    &\quad\quad\;\;\begin{array}{@{}rll}
       x_{2}\in\underset{x_{2}}{\text{argmin}}  & c_{1}^{\top}x_{1}^{\star} + c_{2}^{\top}x_{2}\\
       \st& A_{1}x_{1}^{\star}\geqslant B_{1}S_{1}\hat{\xi}\\
       & A_{1}x_{1}^{\star} + A_{2}x_{2}\geqslant B_{2}S_{2}\hat{\xi},
    \end{array}\tag{LL}\label{LL}
\end{align}
where the upper-level problem \eqref{UL} maximizes the distance between the objective function values of the LDR and scenario-based solution by optimizing the planning uncertainty realization $\hat{\xi}$. Here, we assume that planning uncertainty realizations are within the minimum and maximum bounds $\underline{\xi}$ and $\overline{\xi}$, respectively. Observe, that the LDR matrices $X_{1}^{\star}$ and $X_{2}^{\star}$ are fixed to the optimal solution of problem \eqref{prog:bound_apprx_primal} and that the first-stage decision $x_{1}^{\star}$ is fixed to the solution of problem \eqref{prog:saa_small}. The recourse decision $x_{2}$ is computed in the lower-level problem \eqref{LL}, which takes realization $\hat{\xi}$ as input and computes the least-cost recourse solution $x_{2}$. Therefore, using the closed-loop optimization \eqref{UL}--\eqref{LL}, we can identify the worst-case uncertainty realization scenario, which yields the largest gap between the LDR and scenario-based solutions. Then, using the distributional information in \eqref{eq:amb_set}, we can conclude on the likelihood of the worst-case optimality.  

\section{Numerical Tests for the Southeastern U.S.} \label{sec:case_study}
\begin{figure}
    \centering
    \resizebox{0.46\textwidth}{!}{%
    {%
    \setlength{\fboxsep}{0pt}%
    \fbox{
    \includegraphics{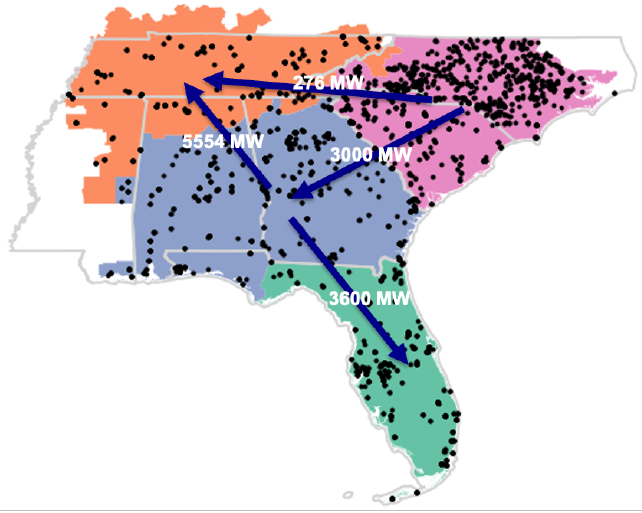}}
    }%
    }
    \textcolor{maincolor}{\caption{Southeastern U.S. layout for carbon-constrained power generation and storage investment planning. The dots depict the nodes of electricity production and generation, then aggregated to four regions as in \cite{schwartz2021role}.}
    \label{fig:my_label}}
\end{figure}

In numerical tests, we use a power system model from \cite{schwartz2021role} resembling seven states in the Southeastern U.S. responsible for $\approx20\%$ of national electric power emissions. The power supply is from regulated utilities, which makes this region suitable for centralized emission-oriented investment planning. In this line, we consider a 25-year long planning horizon until 2050 with five investment stages from 2025 to 2045, modeling five years of operations in between. The full problem data is available in the e-companion \cite{dvorkin2022}; in short, \textcolor{maincolor}{operations are modeled using $14$ operating horizons per year with $24$ representative hours in each}. The current generation mix includes 31 generator in $4$ aggregation zones interconnected by $4$ transmission lines. We optimize investments into 36 generation units, including renewable, nuclear, and gas-fired power generation (with{} and without carbon capture technology), and investments into $4$ utility-scale lithium-ion battery systems, one per zone. We consider that the peak demand, investment costs, and fuel prices are uncertain from 2030 onward, i.e., for a vector of random variables $\boldsymbol{\xi}\in\mathbb{R}^{n}$, we have $n=\sum_{t=1}^{5} n_{t}=13$, with $n_{1}=1$ and $n_{t}=3, \forall t=2,\dots,5$. The ambiguity set $\mathcal{P}$ in \eqref{eq:amb_set} contains any distribution with the first and second moments $\mu=\mathbb{1}$ and $\Sigma=\text{diag}[\sigma^2\cdot\mathbb{1}]$, respectively, with variance $\sigma^2$ uniformly set to 0.25. Figure \ref{fig:uncertainty_data} depicts the resulting planning uncertainty for normally distributed $\boldsymbol{\xi}$. The goal is to accommodate this uncertainty while meeting the CO$_2$ emission targets from Tab. \ref{tab:em_goal}. To guarantee feasibility, we set violation probabilities to $\overline{\varepsilon}^{\text{f}}=12.5\%$ for power flows, $\overline{\varepsilon}^{\text{g}}=\overline{\varepsilon}^{\text{r}}=1\%$ for generator limits,  $\overline{\varepsilon}^{\text{s}}=4\%$ for energy storage limits, $\varepsilon^{\text{e}}=20\%$ for annual CO$_2$ limits, and $\overline{\varepsilon}^{\text{i}}=5.0\%$ for investment limits, unless stated otherwise. 

Table \ref{tab:complexity} collects the resulting dimensions of the deterministic, LDR and scenario approximation problems. We run numerical tests using the JuMP optimization package in Julia language \cite{dunning2017jump}, Mosek optimization solver, and MIT SuperCloud high performance computational (HPC) environment \cite{reuther2018interactive}. \textcolor{maincolor}{Observe, the dimension of the primal and dual LDR approximations, which grow linearly in $\boldsymbol{\xi}$, is by several orders of magnitude smaller than the scenario-based counterpart, which grows exponentially in $\boldsymbol{\xi}$. As a result, if the primal and dual LDR approximations are solved in less than one hour each, the solution to the scenario approximation is not retrieved within the allocated 48-hour running time on HPC.} To replicate the models, we refer to the data and code in \cite{dvorkin2022}.

\begin{figure}
\centering
\includegraphics[width=0.5\textwidth]{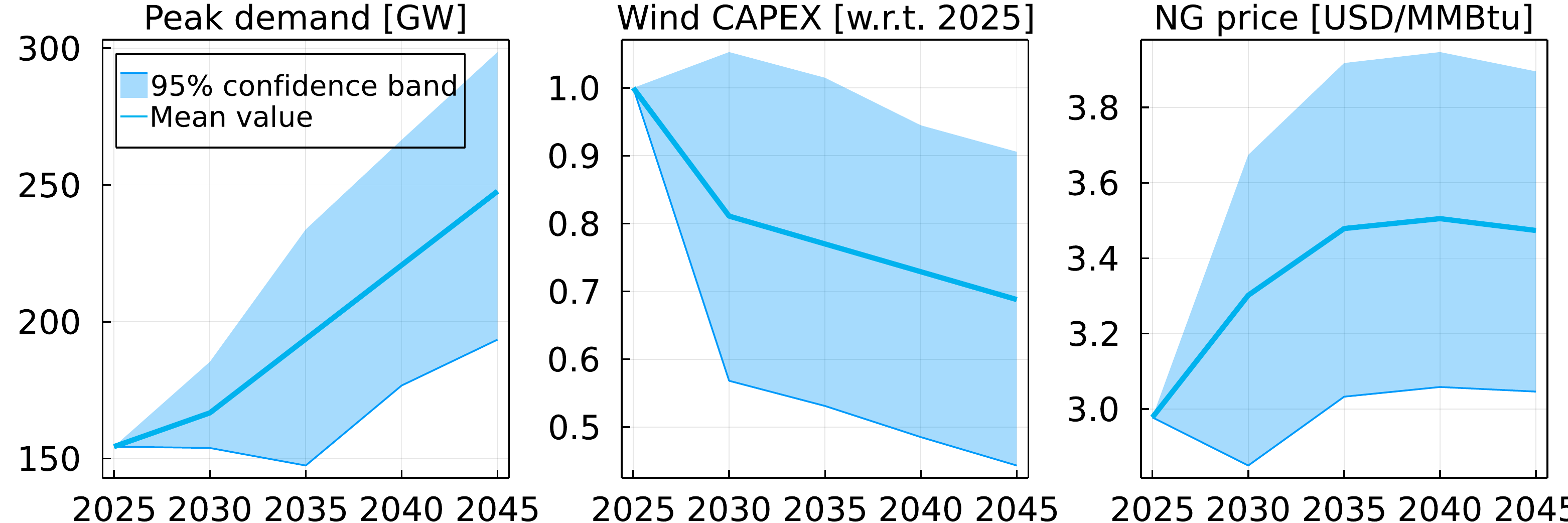}
\caption{Uncertainty of selected planning parameters:
demand data is from NREL electrification studies \cite{murphy2021electrification}, wind CAPEX data is from NREL ATB 2021 \cite{vimmerstedt2021annual}, and gas prices are from the annual IEA energy outlook 2021 \cite{iea2021}.}
\label{fig:uncertainty_data}
\end{figure}

\begin{table}
\caption{Annual emission limit}
\label{tab:em_goal}
\centering
\begin{threeparttable}
\begin{tabular}{c|c|ccccc}
\toprule
Year &
2020\tnote{*}&
2025&
2030&{}
2035&
2040&
2045\\
\midrule
$\overline{e}_{t}$, Mt&
325&150 & 125 & 100 & 75 & 50\\
\bottomrule
\end{tabular}
\begin{tablenotes}\footnotesize
\item[*]Estimated emission level in the Southeastern U.S. \cite{schwartz2021role}
\end{tablenotes}
\end{threeparttable}
\end{table}

\begin{table}
\caption{Investment problems' dimension \textcolor{maincolor}{and CPU time}}
\label{tab:complexity}
\centering
\begin{threeparttable}
\def\arraystretch{1.15}
\begin{tabular}{l|llll}
\toprule
Model & 
Det.&
LDR--P&
LDR--D&
Scen. approx.\tnote{$\star$}
\\
\midrule
Type & LP\tnote{\dag} & SOCP\tnote{\ddag} & SOCP\tnote{\S} & LP\\
\# variables & $132,940$ & $8,130,218$ & $9,861,207$ & $1,130$ mil.\\
\# constraints & $377,071$ & $7,036,498$ & $7,153,874$ & $11,277$ bil.\\
\multicolumn{1}{r|}{-- linear} & $377,071$ & $6,541,451$ & $6,642,883$ & $11,277$ bil.\\
\multicolumn{1}{r|}{-- conic} & $-$ & $495,047$ & $510,991$ & $-$\\
\textcolor{maincolor}{CPU time, sec.} & \textcolor{maincolor}{13.4} & \textcolor{maincolor}{885.6} & \textcolor{maincolor}{3,318.6} & \textcolor{maincolor}{$\infty$}\\
\bottomrule
\end{tabular}
\begin{tablenotes}\footnotesize
\item[]\tnote{\dag}LP -- linear program, \tnote{\ddag}SOCP -- second-order cone program
\item[]\tnote{\S}LDR-P/LDR-D -- Primal/Dual LDR approximation
\item[]\tnote{$\star$}Using 10 uncertainty scenarios (branches) per stage
\end{tablenotes}
\end{threeparttable}

\end{table}

\subsection{Investment Results}

First, we compare the deterministic investment plan and its stochastic counterparts in Tab. \ref{tab:summary} using the following metrics:
\begin{itemize}
    \item In-sample cost -- the total optimization planning cost as the optimal value of functions \eqref{eq_obj} and \eqref{prog:sto_obj} (line $\star$).
    \item The mean out-of-sample (O--of--S) cost (lines $\diamond$) across $1,000$ scenarios drawn from four different distributions (Normal, Uniform, Logistic and Laplace) that match the mean and covariance of set $\mathcal{P}$. For each scenario, the cost is obtained using re-optimization similar to problem \eqref{prog:base}, where the investment plan is fixed, e.g., to $\overline{y}_{t}^{\star}$ for deterministic solution and to $\overline{Y}_{t}^{\star}S_{t}\hat{\xi}$ for stochastic solution, where $\hat{\xi}$ is an uncertainty realization scenario, and where planning parameters are adjusted according to $\hat{\xi}$. If the investment plan fails to satisfy a particular demand realization, we model a variable load shedding that incurs an additional cost of $\$9,000/$MWh.  
    \item The frequency and the mean magnitude of load shedding across the investment horizon (lines $\triangleleft$), by stress-testing investment plans on $1,000$ out-of-sample scenarios.
    \item Dissimilarity in the 1$^{\text{st}}-$stage deterministic and stochastic decisions using $\ell_{1}-$norm, i.e., $\norm{\overline{y}_{1} - \overline{Y}_{1}S_{1}\mathbb{1}}_{1}$ for generation, and similarly for storage investments (lines $\bullet$).
    \item The standard deviation of the investment plan (lines $\ast$).
\end{itemize}

\bgroup
\def\arraystretch{1.15}
\begin{table*}[t]
\caption{Summary of deterministic and stochastic planning under Normal (Norm. asm.) and distributionally robust (DRO) assumptions.}
\label{tab:summary}
\centering
\begin{tabular}{lllrrrrrrr}
\toprule
\multicolumn{2}{c}{\multirow{4}{*}{Parameter}} & \multirow{4}{*}{Unit} & \multicolumn{1}{c}{\multirow{4}{*}{\begin{tabular}[c]{@{}c@{}}Deterministic\\ solution\end{tabular}}} & \multicolumn{6}{c}{Stochastic chance-constrained  LDR optimization} \\
\cmidrule(lr){5-10}
\multicolumn{2}{c}{} &  & \multicolumn{1}{c}{} & \multicolumn{2}{c}{\multirow{2}{*}{Base LDR solution}} & \multicolumn{4}{c}{Quasi-deterministic solution} \\
\cmidrule(lr){7-10}
\multicolumn{2}{c}{} &  & \multicolumn{1}{c}{} & \multicolumn{2}{c}{} & \multicolumn{2}{c}{$\alpha^{\overline{y}}=\alpha^{\overline{\vartheta}}=\alpha^{\overline{\varphi}}=10\%$} & \multicolumn{2}{c}{$\alpha^{\overline{y}}=\alpha^{\overline{\vartheta}}=\alpha^{\overline{\varphi}}=0.1\%$} \\
\cmidrule(lr){5-6}\cmidrule(lr){7-8} \cmidrule(lr){9-10}
\multicolumn{2}{c}{} &  & \multicolumn{1}{c}{} & \multicolumn{1}{c}{Norm. asm.} & \multicolumn{1}{c}{DRO} & \multicolumn{1}{c}{Norm. asm.} & \multicolumn{1}{c}{DRO} & \multicolumn{1}{c}{Norm. asm.} & \multicolumn{1}{c}{DRO} \\
\midrule
\multicolumn{2}{l}{$\star$ In-sample cost} & bil. USD & 526.2 & 529.6 & 635.1 & 551.5 & 685.5 & 578.1 & 895.6 \\
\midrule
\multirow{4}{*}{\rotatebox{90}{\begin{tabular}[c]{@{}c@{}}O--of--S \\ cost\end{tabular}}} & ${\diamond}$ Normal & \multirow{4}{*}{bil. USD} & 3,317 & 535.5 & 604.8 & 741.1 & 639.7 & 802.7 & 795.4 \\
 & ${\diamond}$ Uniform &  & 3,333 & 532.0 & 604.8 & 698.7 & 638.2 & 752.0 & 792.8 \\
 & ${\diamond}$ Logistic &  & 3,369 & 536.2 & 609.1 & 772.5 & 641.1 & 885.9 & 795.7 \\
 & ${\diamond}$ Laplace &  & 3,365 & 542.0 & 609.8 & 794.3 & 642.0 & 956.4 & 791.7 \\
\midrule
\multirow{4}{*}{\rotatebox{90}{\begin{tabular}[c]{@{}c@{}}Load shed. \\ freq. (mag.)  \end{tabular}}} & $\triangleleft$ Normal & \multirow{4}{*}{\% (GWh)} & 76.9 (4,560) & 51.8 (17.8) & 0.1 (0.0) & 31.9 (408) & 0.1 (0.0) & 28.8 (442) & 0.0 (0.0) \\
 & $\triangleleft$ Uniform &  & 78.7 (5,116) & 61.3 (7.6) & 0.0 (0.0) & 34.8 (344) & 0.0 (0.0) & 31.7 (377) & 0.0 (0.0) \\
 & $\triangleleft$ Logistic &  & 79.3 (4,601) & 47.9 (17.1) & 0.1 (0.0) & 29.5 (434) & 0.2 (0.0) & 24.4 (559) & 0.1 (0.6) \\
 & $\triangleleft$ Laplace &  & 79.3 (4,404) & 38.6 (30.7) & 0.7 (1.8) & 26.6 (478) & 0.4 (4.0) & 23.9 (656) & 0.1 (0.5) \\[.25em]
\midrule
\multicolumn{2}{l}{$\bullet$ Dissimilarity $ \overline{Y}_{1}$} & GW & 0.0 & 0.8 & 26.1 & 3.6 & 28.0 & 6.4 & 36.7 \\
\multicolumn{2}{l}{$\bullet$ Dissimilarity $ \overline{\Theta}_{1}$} & GWh & 0.0 & 0.0 & 18.9 & 0.0 & 17.7 & 0.0 & 7.7 \\
\midrule
\multicolumn{2}{l}{$\ast$ $\Sigma_{i,t}\text{Std}[\overline{y}_{ti}(\xi)]$} & GW & 0.0 & 106.8 & 68.0 & 34.5 & 45.7 & 0.35 & 0.55 \\
\multicolumn{2}{l}{$\ast$ $\Sigma_{i,t}\text{Std}[\overline{\vartheta}_{ti}(\xi)]$} & GWh & 0.0 & 28.7 & 17.8 & 5.6 & 10.5 & 0.05 & 0.24 \\
\multicolumn{2}{l}{$\ast$ $\Sigma_{i,t}\text{Std}[\overline{\varphi}_{ti}(\xi)]$} & GW & 0.0 & 11.6 & 13.0 & 2.4 & 8.5 & 0.02 & 0.44 \\
\bottomrule
\end{tabular}
\end{table*}
\egroup

The cost data in Tab. \ref{tab:summary} reveals the impact of modeling assumptions on the multi-stage investments performance. The in-sample cost of the deterministic solution (3$^{\text{rd}}$ column) is the lowest as it ignores the planning uncertainty, while the stochastic optimization (4$^{\text{th}}$ and 5$^{\text{th}}$ columns) demonstrates larger in-sample costs that additionally include the cost of uncertainty. Under the normal distributional assumption\footnote{
This assumption leverages a common and less conservative reformulation, where the rules are optimized using a single-sided reformulation of chance constraints, as in \eqref{eq:single_sided_ref}, but the safety factor $\tilde{\varepsilon}$ amounts to the inverse of the the inverse CDF of the standard Normal distribution at $(1-\varepsilon)-$quantile, which is smaller than distributionally robust requirement $\sqrt{(1-\varepsilon)/\varepsilon}$. 
}, the cost of uncertainty is smaller than in the distributionally robust case, which optimizes the planning cost for the worst-case distribution (rather than for the normal one). On the other hand, by factoring in the planning uncertainty, the stochastic investment plans outperform the deterministic plan in terms of the out-of-sample cost and the load shedding frequency and magnitude: the deterministic out-of-sample cost exceeds the in-sample one by an order of magnitude due to frequent and large load shedding. However, the out-of-sample statistics under the normal assumption are not as good as those in the distributionally robust planning. If for the two-stage optimization, the normal assumption fairs well in practice \cite{bienstock2014chance}, in the multi-stage optimization, where the number of constraints grows in investment stages, it may not be the case. The distributionally robust solution, however, improves the out-of-sample performance of the stochastic multi-stage plan.

\begin{figure*}[t]
\centering
\resizebox{0.95\textwidth}{!}{%
\includegraphics[width=0.48\textwidth]{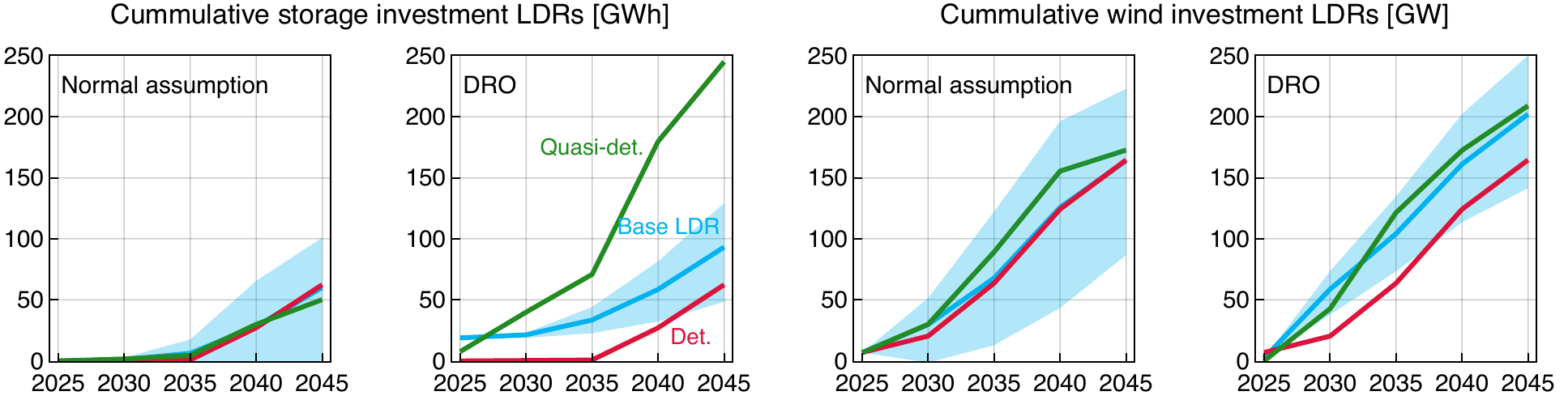}
}
\caption{Cumulative storage energy capacity (left plots) and wind generation (right plots)  LDR-guided investments under different distributional assumptions. The blue envelopes depict the 95\% confidence band of the base stochastic investment portfolio. The quasi-deterministic solution is given for $\alpha^{\mydot}=0.1\%$.}
\label{fig:ldr_plots}
\end{figure*}

The dynamic investments into energy storage and wind generation are illustrated in Fig. \ref{fig:ldr_plots}. Observe how uncertainty and distributional assumptions affect the $1^{\text{st}}-$stage investments (also through lines $\bullet$ of Tab. \ref{tab:summary}), and how the base LDR solution accumulates large variance of the stochastic investment portfolio. For example, the spread of wind investment decisions already at the $2^{\text{nd}}-$stage is $\approx$50 GW,  regardless of distributional assumptions. To reduce investment stochasticity, consider the variance-constrained optimization from Section \ref{sec:inv_var_red_and_det_eq} and with results  in $7^{\text{th}}$ to $10^{\text{th}}$ columns in Tab. \ref{tab:summary}. Observe, that the variance reduction (or standard deviation in lines $\ast$) comes at the expense of the increasing in-sample planning cost (line $\star$), thus establishing the trade-off between the cost and stochasticity of the investment plan. Notably, the variance reduction systematically reduces the load shedding frequency (6$^{\text{th}}$ and 7$^{\text{th}}$ columns) as opposed to the base LDR solution (4$^{\text{th}}$ column), but it does not necessarily reduce its magnitude (recall, chance constraints control the frequency of violations, not their magnitudes). As a result, the out-of-sample costs are substantially higher than the in-sample costs under the Normal assumption. The DRO optimization, on the other hand, systematically demonstrates both a small frequency and a small magnitude of load shedding. With the smallest coefficients $\alpha^{\mydot}=0.1\%$, the variance-constrained LDR optimization returns a quasi-deterministic investment plan, also depicted in Fig. \ref{fig:ldr_plots} in green. The quasi-deterministic investment plan is insensitive to uncertainty realizations, but unlike the deterministic solution, it identifies that minimal deterministic level of investments that guarantees the prescribed operational and policy constraint satisfaction, hence resulting in more robust out-of-sample performance. 

\begin{figure*}[t]
\centering
\resizebox{1\textwidth}{!}{%
\includegraphics[width=0.48\textwidth]{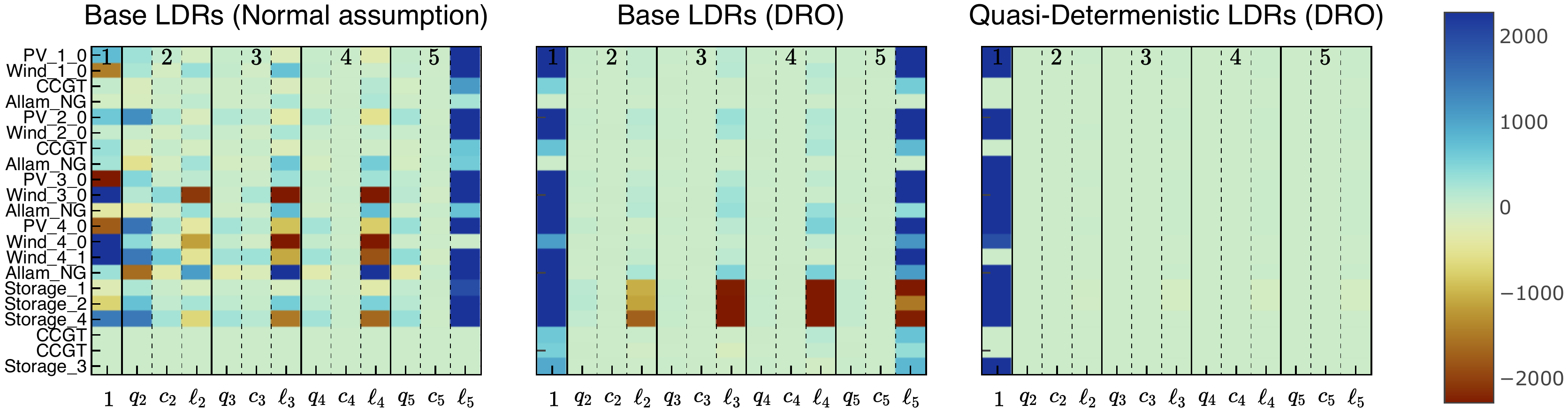}
}
\caption{Visualization of the optimized, non-zero entries of generation ($\overline{Y}_{5}$) and storage energy capacity ($\overline{\Theta}_{5}$) LDRs at the terminal $5^{\text{th}}$ stage (2045) under different assumptions. The color bar depicts the sensitivity (in MW for generation and MWh for storage energy capacity) of investments to the realizations of 12 random variables modeling CAPEX $(q_{t})$, fuel cost $(c_{t})$ and peak load $(\ell_{t})$ uncertainties at the preceding and current investment stages.}
\label{fig:ldr_matrices}
\end{figure*}

To illustrate the discrepancies in stochastic investment optimization, Fig. \ref{fig:ldr_matrices} displays the sensitivities of the optimal investments to the CAPEX, fuel price, and peak load uncertainties under the normal assumption, distributional robustness, and in the quasi-deterministic case. With larger conservatism (in terms of the in-sample costs, line $\star$ in Tab. \ref{tab:summary}), the optimal investments become less sensitive to the uncertain realizations of planning parameters. Notably, sensitivities to CAPEX and fuel prices are substantially smaller than the sensitivity to the peak demand realizations. The DRO solution is less sensitive due to a substantially larger safety factor $\tilde{\varepsilon}$ in \eqref{eq:single_sided_ref} than that under the normal assumption, which more substantially reduces the feasible space of the investment LDRs. The almost zero sensitivity of the quasi-deterministic solution is due to additional constraints in \eqref{prog:var_con} on the LDR standard deviation.   

Last, with Fig. \ref{tab:emission}, we illustrate the ability of the chance-constrained LDR optimization to trade-off between the investment feasibility (here, carbon policy feasibility) and the expected planning cost. On the left plot, the LDRs are optimized in a distributionally robust fashion to maintain at least $1-\varepsilon^{\text{e}}=80\%$ of the annual emission probability mass below the target emission level of 50 Mt. By increasing this requirement to 99\%, the model's conservatism substantially increases and the probability mass shifts to the left from the target level, increasing the expected cost from 635.1 to 669.9 bil. USD. Such a drastic shift is due to the model's robustness against any possible distributional shape matching the given mean $\mu$ and covariance $\Sigma$, including the long-tail distributions. 

\begin{figure}
\centering
\includegraphics[width=0.48\textwidth]{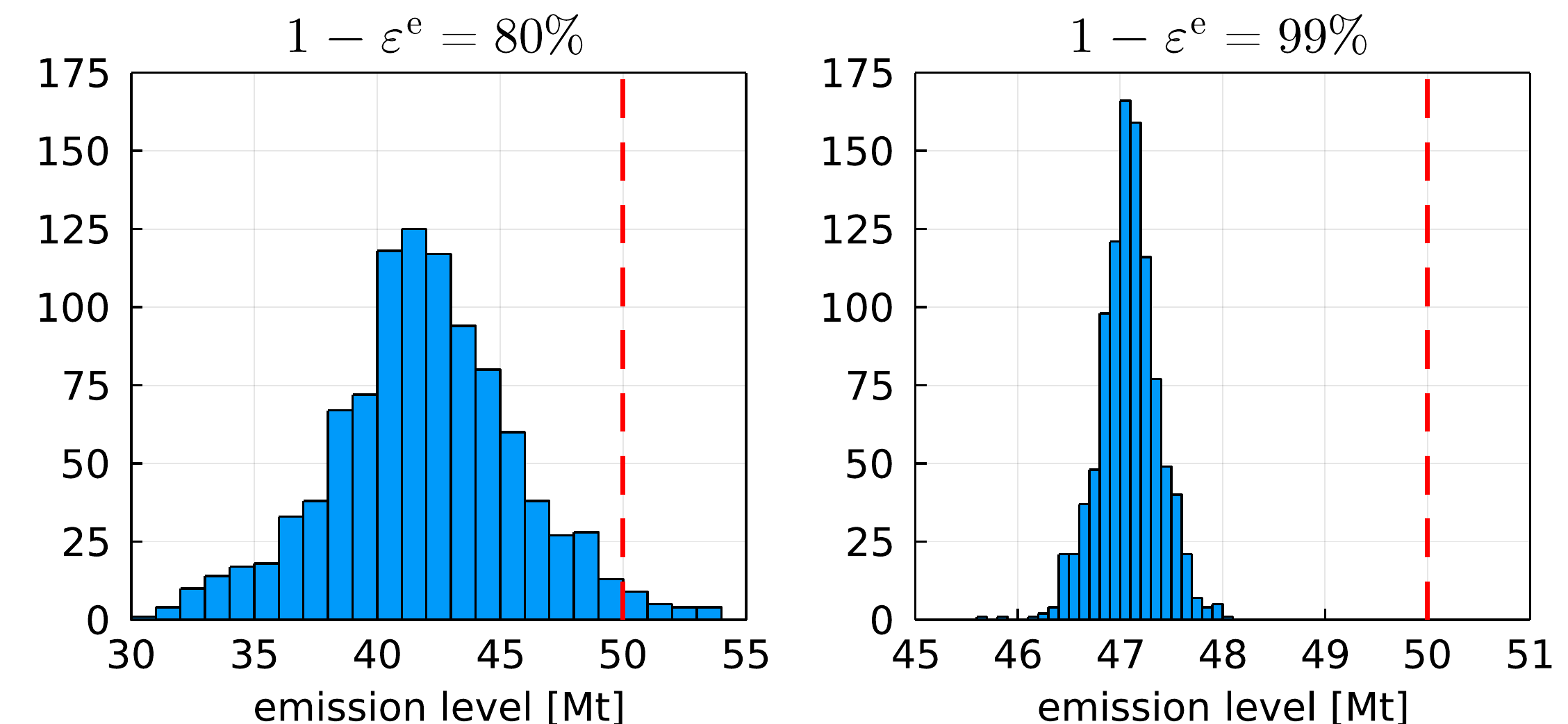}

\vspace{0.25cm}
\resizebox{0.45\textwidth}{!}{%
\begin{tabular}{l|ccccc}
\toprule
Probability $1-\varepsilon^{\text{e}},\%$&
80&
90&
95&
99&
99.9\\
\midrule
In-sample cost, bil. USD&
635.1&
640.8&
647.9&
668.3&
669.9\\
\bottomrule
\end{tabular}
}
\vspace{0.25cm}
\caption{Trade-offs between the CO$_2$ policy constraint satisfaction and the in-sample planning costs under varying risk tolerances in distributionally robust optimization. The plots display distributions of the CO$_2$ emission level in 2045 for $\varepsilon^{\text{e}}=0.2$ and $\varepsilon^{\text{e}}=0.01$ using $10^3$ scenarios from Laplace distribution, where the vertical dashed lines depict the CO$_2$ emission limit for 2045.}
\label{tab:emission}
\end{figure}

\subsection{Analysis of the LDR Sub-Optimality}
\begin{table*}[t]
\caption{Global LDR sub-optimality bounds for varying constraint violation probabilities and distributional assumptions}
\label{tab:gap}
\centering
\def\arraystretch{1.15}
\begin{tabular}{ll|lllll|lllll}
\toprule
\multicolumn{2}{l}{Distributional assumption} & \multicolumn{5}{c}{Normal} & \multicolumn{5}{c}{DRO} \\
\cmidrule(lr){3-7}\cmidrule(lr){8-12}
\multicolumn{2}{l}{Variance of $\boldsymbol{\xi}$, $\sigma^2$} & \multicolumn{1}{c}{0.05} & \multicolumn{1}{c}{0.10} & \multicolumn{1}{c}{0.15} & \multicolumn{1}{c}{0.20} & \multicolumn{1}{c}{0.25} & \multicolumn{1}{c}{0.05} & \multicolumn{1}{c}{0.10} & \multicolumn{1}{c}{0.15} & \multicolumn{1}{c}{0.20} & \multicolumn{1}{c}{0.25} \\
\midrule
\multicolumn{12}{c}{Theoretical constraint violation probability $\hat{\varepsilon}= 10\%$} \\
\midrule
Prim. objective $\overline{\text{P}}(X^{\star})$ & bil. USD & 526.3 & 526.4 & 526.6 & 526.7 & 526.8 & 564 & 565.7 & 567.8 & 570.4 & 573.8 \\
Dual\textcolor{white}. objective $\overline{\text{D}}(\Lambda^{\star})$ & bil. USD & 525.7 & 524.7 & 524.1 & 523 & 522.1 & 521 & 513.4 & 510.7 & 506.2 & 501.7 \\
Absolute difference & bil. USD & 0.6 & 1.69 & 2.5 & 3.7 & 4.7 & 43 & 52.3 & 57.1 & 64.2 & 72.1 \\
Percentage difference & \% & 0.1 & 0.3 & 0.5 & 0.7 & 0.9 & 7.6 & 9.3 & 10.1 & 11.3 & 12.6 \\
\midrule
\multicolumn{12}{c}{Theoretical constraint violation probability $\hat{\varepsilon}= 5\%$} \\
\midrule
Prim. objective $\overline{\text{P}}(X^{\star})$ & bil. USD & 526.5 & 526.9 & 527.4 & 527.8 & 528.2 & 566 & 571.2 & 579.5 & 589.3 & 598.7 \\
Dual\textcolor{white}. objective $\overline{\text{D}}(\Lambda^{\star})$ & bil. USD & 525.5 & 523.5 & 521.8 & 520.1 & 518.5 & 510.5 & 505.1 & 495.7 & 486.5 & 477.2 \\
Absolute difference & bil. USD & 1 & 3.4 & 5.6 & 7.7 & 9.7 & 55.5 & 66.1 & 83.8 & 102.8 & 121.5 \\
Percentage difference & \% & 0.2 & 0.7 & 1.1 & 1.5 & 1.8 & 9.8 & 11.6 & 14.5 & 17.4 & 20.3 \\
\midrule
\multicolumn{12}{c}{Theoretical constraint violation probability $\hat{\varepsilon}= 2.5\%$} \\
\midrule
Prim. objective $\overline{\text{P}}(X^{\star})$ & bil. USD & 526.8 & 527.4 & 528.1 & 528.8 & 529.6 & 571.5 & 590.3 & 609.3 & 626.8 & 643.3 \\
Dual\textcolor{white}. objective $\overline{\text{D}}(\Lambda^{\star})$ & bil. USD & 522.6 & 521.9 & 516.9 & 517.3 & 510.3 & 504.6 & 485.5 & 465 & 429.2 & 401.2 \\
Absolute difference & bil. USD & 4.2 & 5.5 & 11.2 & 11.5 & 19.3 & 66.9 & 104.8 & 144.3 & 197.6 & 242.1 \\
Percentage difference & \% & 0.8 & 1.0 & 2.1 & 2.2 & 3.6 & 11.7 & 17.8 & 23.7 & 31.5 & 37.6\\
\bottomrule
\end{tabular}
\end{table*}

We now analyze the sub-optimality of the investment LDRs using the two methodologies from Section \ref{sec:sub-optimality_gurantees}. First, we solve the primal and dual LDR approximations \eqref{prog:bound_apprx_primal} and \eqref{prog:bound_apprx_dual}, respectively. Then, according to Theorem \ref{th:gap}, we obtain the global sub-optimality bound on investment LDRs by taking the difference between the optimal primal and dual objective function values. We run this experiment for the LDR optimization under the normal and distributionally robust assumptions, using varying constraint violation probability $\hat{\varepsilon}$ and varying distribution variance $\sigma^2$, and report the results in Tab. \ref{tab:gap}. 
Observe, that the sub-optimality bound growths as the distribution variance increases. However, as distributional assumptions and violation probabilities are prescribed inputs to the LDR optimization, one obtains a tighter sub-optimality bound (guarantee) by changing the risk appetite, i.e., by increasing input parameter $\hat{\varepsilon}$. Notably, the absolute sub-optimality in the bil. USD  remains substantially smaller than the difference between the out-of-sample costs of the deterministic and stochastic solutions (see Table \ref{tab:summary}), meaning that even though the LDR is sub-optimal, the system planning is still better off with the LDR solution compared to the uncertainty-agnostic, deterministic solution. 

While the global sub-optimality bound guarantees the average LDR performance (i.e., expected optimality loss), the actual sub-optimality will depend on the uncertainty realization scenario. To identify the likelihood of the worst-case sub-optimality scenario, we solve a scenario-based stochastic program \eqref{prog:saa_small} for a reduced problem size, which includes two investment stages 2025 and 2045, models the 2$^{\text{nd}}$-stage peak load uncertainty $\ell_{2}(\xi)$ only, disregards energy storage expansion, and reduces the number of operating horizons from 14 to 5. Using 50 uncertainty load scenarios from a Normal distribution, the $1^{\text{st}}-$stage investment solution is retrieved in less than 3 hours. We then use this solution to solve a bilevel problem \eqref{UL}--\eqref{LL}, where the minimum and maximum support bounds, $\underline{\xi}$ and $\overline{\xi}$ respectively, on the worst-case scenario $\hat{\xi}$ are obtained from sampling 1000 scenarios from the normal distribution of $\boldsymbol{\xi}$. The results for the various standard deviations of load uncertainty are reported in Tab. \ref{tab:wc_scenario}. Observe, that as the standard deviation increases, the minimum and maximum bounds also increase. The optimal load uncertainty scenarios $\hat{\xi}^{\star}$, i.e., $\ell_{2}(\hat{\xi}^{\star})$, turn out to lie at the boundary of the distribution support. Thus, we conclude that the likelihood of the worst-case LDR sub-optimality is very small, as such scenarios pertain to distribution tails. 

\begin{table}[]
\centering
\caption{Results of the worst-case sub-optimality scenario learning}
\label{tab:wc_scenario}
\setlength\tabcolsep{4pt}
\def\arraystretch{1.15}
\begin{tabular}{l|lllllll}
\toprule
\multicolumn{2}{l}{\begin{tabular}[c]{@{}l@{}}2$^{\text{nd}}$-stage peak load \\ standard deviation\end{tabular}} & GW & 0.0 & 9.8 & 21.4 & 29.8 & 34.9 \\
\midrule
\multicolumn{2}{l}{Minimum\hspace{0.15em} peak load} & GW & 248.4 & 217.8 & 188.9 & 151.8 & 122.0 \\
\multicolumn{2}{l}{Maximum peak load} & GW & 248.4 & 288.5 & 324.0 & 353.2 & 336.9 \\
\multirow{2}{*}{Worst-case} & Scenario & GW & 248.4 & 288.5 & 324.0 & 353.1 & 336.9 \\
 & Opt. loss & bil. USD & 0.0 &  0.7 & 11.3 & 24.9 & 52.4 \\
\bottomrule
\end{tabular}
\end{table}

\section{Conclusions} \label{sec:conclusions}

We revisited the application of multi-stage linear decision rules to power system expansion problems and proposed a novel chance-constrained optimization with performance guarantees. Using a realistic model of the U.S. Southeast’s power system, we demonstrated decarbonization-oriented investment planning under uncertainty for several decades ahead. The chance-constrained formulation enabled the robustness of investments to distributional ambiguity, while offering the trade-offs between the cost of investments and their operational and policy feasibility. Moreover, we improved on the standard deterministic power system planning by offering a quasi-deterministic planning model, which identifies the necessary corrections to the fixed investment plans to immunize investment decisions and future system operations against uncertainty. We improved on previous studies by certifying the economic efficiency of the investment LDRs through the global sub-optimality bound and by studying the likelihood of the worst-case sub-optimality realization scenarios.

\bibliography{references.bib}
\bibliographystyle{IEEEtran}

\begin{IEEEbiography}[{\includegraphics[width=1in,height=1.25in,clip,keepaspectratio]{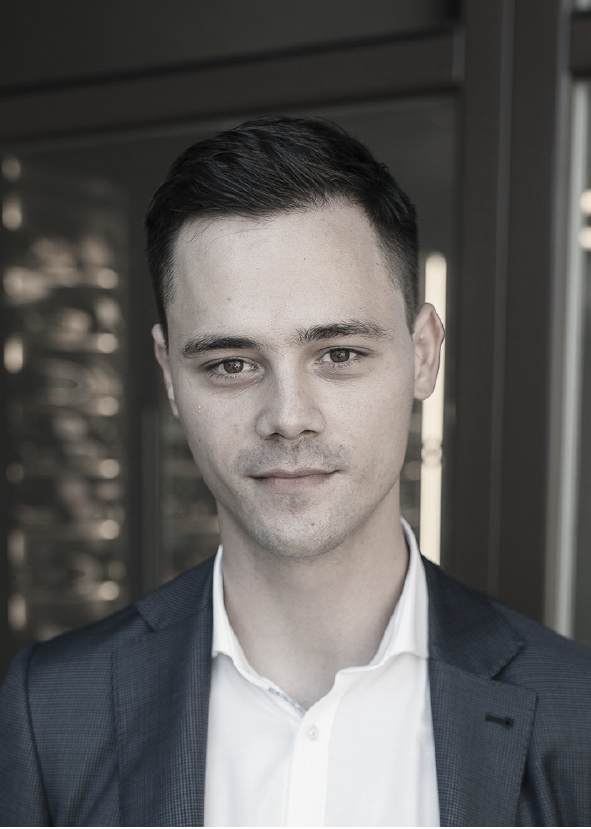}}]{Vladimir Dvorkin Jr.} (S’18, M'21) is a postdoctoral fellow at the Massachusetts Institute of Technology (MIT) affiliated with the Energy Initiative and Laboratory for Information \& Decision Systems. Before joining MIT, he earned a PhD degree in electrical engineering at the Technical University of Denmark in 2021, and was a visiting scholar at the School of Industrial and Systems Engineering of Georgia Tech in 2019. He studies the energy transition through the lens of decision-making under uncertainty, energy economics, and algorithmic privacy. 
His work is recognized by several prestigious awards, including the Marie Skłodowska-Curie Actions \& Iberdrola Group postdoctoral fellowship and the IEEE Transactions on Power Systems Best Paper Award.
\end{IEEEbiography}

\begin{IEEEbiography}[{\includegraphics[width=1in,height=1.25in,clip,keepaspectratio]{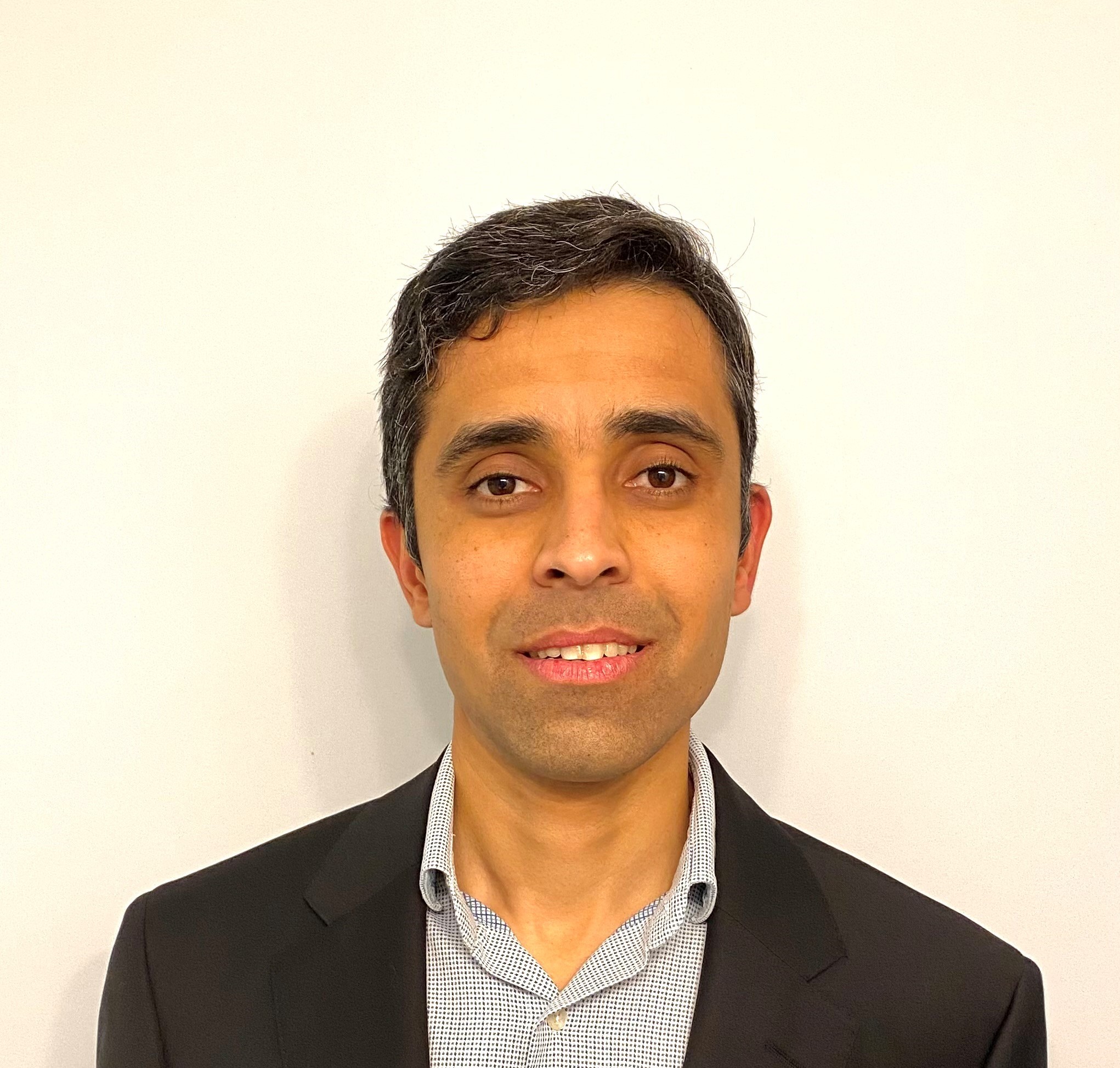}}]{Dharik Mallapragada} is a Principal Research Scientist at the MIT Energy Initiative (MITEI), where he leads the Sustainable Energy Transitions Group. Dr. Mallapragada’s research focuses on planning and operating resilient, low-carbon energy systems as well as conceptualization, design and integration of emerging energy technologies. At MIT, he has pursued research in these topics while securing funding from government, industry and philanthropic sources and establishing collaboration with multiple principal investigators across MIT and other institutions. Prior to MIT, Dr. Mallapragada spent nearly five years in the energy industry working on a range of sustainability-focused research topics. He recently served as a member of the Massachusetts Commission on Clean Heat, and serves on the advisory committee for the Open Energy Outlook project, a multi-institution effort to create open-source energy systems models and datasets. He also co-leads systems thrust activities at the Center for Decarbonizing Chemical Manufacturing using Sustainable electrification (DC-MUSE). Dr. Mallapragada holds a M.S. and Ph.D. in Chemical Engineering from Purdue University and a B.Tech. in Chemical Engineering from the Indian Institute of Technology, Madras, India. 
\end{IEEEbiography}

\begin{IEEEbiography}[{\includegraphics[width=1in,height=1.25in,clip,keepaspectratio]{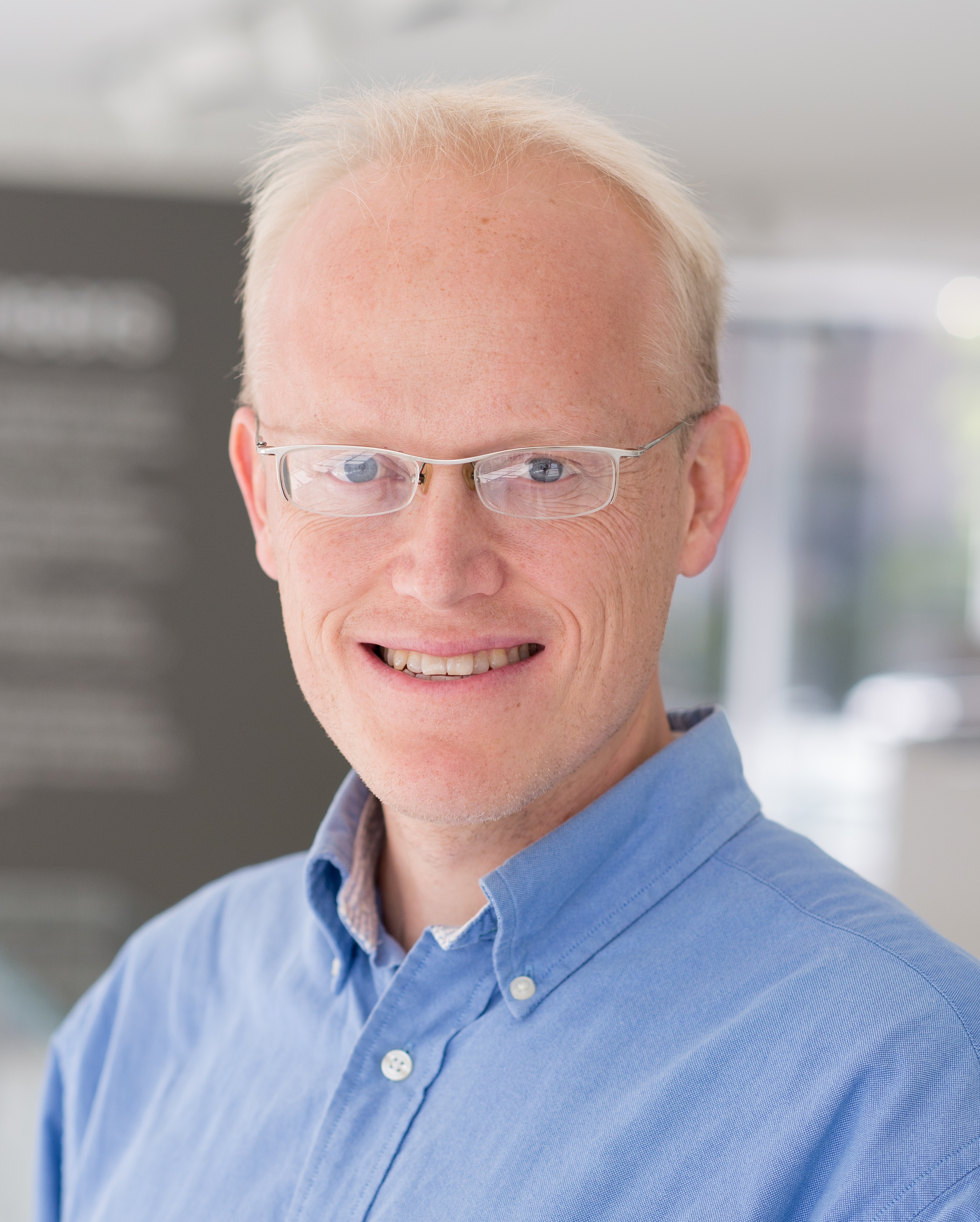}}]{Audun Botterud} (M'05) is a Principal Research Scientist in Laboratory for Information and Decision Systems (LIDS) at MIT, where he leads the Energy Analytics Group. He has a co-appointment in the Energy Systems and Infrastructure Analysis Division at Argonne National Laboratory. His research interests include power systems, electricity markets, renewable energy, and energy storage. Audun holds a M.Sc. (Industrial Engineering) and a Ph.D. (Electrical Power Engineering), both from the Norwegian University of Science and Technology. He was previously with SINTEF Energy Research in Trondheim, Norway.
\end{IEEEbiography}

\clearpage
\newpage
\onecolumn
\appendix
\subsection{Tractable Reformulation of the Multi-Stage Chance-Constrained Investment Problem \eqref{prog:sto}}\label{app:full_tract_ref}
Using uncertainty model \eqref{eq:uncertain_data_def}, linear decision rules \eqref{eq:LDRs}, and following the derivations in Section \ref{sec:cc_ldr}, the multi-stage chance-constrained investment problem \eqref{prog:sto} takes the following second-order cone programming form:
\begin{subequations}\label{prog:full_tract_ref}
\begin{align}
\minimize{\overline{\mathcal{V}}}\quad&\sum_{t=1}^{T}\Big(
\text{Tr}[S_{t}\widehat{\Sigma} S_{t} (Q_{t}^{\text{g}\top}\overline{Y}_{t} + 
Q_{t}^{\text{s}\top}\overline{\Theta}_{t} +
Q_{t}^{\text{p}\top}\overline{\Phi}_{t})
]
+o_{t}^{\text{e}\top}\overline{P}_{t}S_{t}\mathbb{1} 
+\sum_{\tau=1}^{t}\Big(
o_{t}^{\text{c}\top}\overline{Y}_{\tau}  + {}
o_{t}^{\text{s}\top}\overline{\Theta}_{\tau} + 
o_{t}^{\text{p}\top}\overline{\Phi}_{\tau}\Big)S_{\tau}\mathbb{1}  \nonumber \\
&\quad\quad+\sum_{w=1}^{W}\omega_{w}\sum_{h=1}^{H}\text{Tr}[
S_{t}\widehat{\Sigma} S_{t} (
C_{t}^\fl{e}P_{twh} + C_{t}^\fl{c}Y_{twh})
]\Big)\\
\st\quad
&
\mathbb{1}^{\top}\big(P_{twh} + Y_{twh} + \Phi_{twh}^{\shortplus} - k_{twh}^{\ell}\circ L_{t} - \Phi_{twh}^{\shortminus} \big) = \mathbb{0},\\
&
\Theta_{twh} - \Theta_{tw(h-1)} - \Phi_{twh}^{\shortplus}\eta^{\shortplus}+ \Phi_{twh}^{\shortminus}/\eta^{\shortminus} = \mathbb{0},\\
&\left.\begin{aligned}
&\norm{
\begin{bmatrix*}
\overline{\Sigma}\left[F(P_{twh} + Y_{twh} + \Phi_{twh}^{\shortminus} - k_{twh}^{\ell} \circ L_{t} - \Phi_{twh}^{\shortplus})S_{t}\right]_{e}^{\top}\\
z_{twhe}^{\overline{f}}
\end{bmatrix*}
}\leqslant
\sqrt{\overline{\varepsilon}^{\fl{f}}}\left(\overline{f}_{e} - x_{twhe}^{\overline{f}}\right),\\
&\left|\left[F(P_{twh} + Y_{twh} + \Phi_{twh}^{\shortminus} - k_{twh}^{\ell} \circ L_{t} - \Phi_{twh}^{\shortplus})S_{t}\right]_{e}\mathbb{1}\right|\leqslant
z_{twhe}^{\overline{f}} + x_{twhe}^{\overline{f}},\\
&\overline{f}_{e} \geqslant x_{twhe}^{\overline{f}} \geqslant 0, z_{twhe}^{\overline{f}} \geqslant 0
\end{aligned}
\right\}\forall e\in\llbracket E\rrbracket,\label{tractable_ref_flow}\\
&
\left.\begin{aligned}
&\norm{
\begin{bmatrix*}
\overline{\Sigma}\left[P_{twh}S_{t}\right]_{i}^{\top}\\
z_{twhi}^{\overline{p}}
\end{bmatrix*}
} \leqslant \sqrt{\overline{\varepsilon}^{\fl{g}}}\left(\tfrac{1}{2}k_{twhi}^{\fl{e}}\overline{p}_{ti} - x_{twhi}^{\overline{p}}\right),\\
&\left|\left[P_{twh}S_{t}\right]_{i}\mathbb{1} - \tfrac{1}{2}k_{twhi}^{\fl{e}}\overline{p}_{ti}\right|\leqslant z_{twhi}^{\overline{p}} + x_{twhi}^{\overline{p}},\\
&\tfrac{1}{2}k_{twhi}^{\fl{e}}\overline{p}_{ti} \geqslant x_{twhi}^{\overline{p}} \geqslant 0, z_{twhi}^{\overline{p}} \geqslant 0
\end{aligned}
\right\}\forall i\in\llbracket N\rrbracket,\label{tractable_ref_gen_ds}\\
&\left.\begin{aligned}
&\norm{\overline{\Sigma}\left[k_{twhi}^{\fl{e}}\circ\textstyle\sum_{\tau=1}^{t}\overline{Y}_{\tau}S_{\tau} - Y_{twh}S_{t}\right]_{i}^{\top}}\leqslant\tfrac{1}{\sqrt{(1-\overline{\varepsilon}^{\fl{g}})/\overline{\varepsilon}^{\fl{g}}}}
\left[k_{twhi}^{\fl{e}}\circ\textstyle\sum_{\tau=1}^{t}\overline{Y}_{\tau}S_{\tau} - Y_{twh}S_{t}\right]_{i}\mathbb{1},\\
&\norm{\overline{\Sigma}\left[Y_{twh}S_{t}\right]_{i}^{\top}}\leqslant\tfrac{1}{\sqrt{(1-\overline{\varepsilon}^{\fl{g}})/\overline{\varepsilon}^{\fl{g}}}}
\left[Y_{twh}S_{t}\right]_{i}\mathbb{1}\\
\end{aligned}
\right\}\forall i\in\llbracket N\rrbracket,\label{tractable_ref_gen_ss}\\
&
\left.\begin{aligned}
&\norm{
\begin{bmatrix*}
\overline{\Sigma}\left[\left(P_{twh}-P_{tw(h-1)}\right)S_{t}\right]_{i}^{\top}\\
z_{twhi}^{r^{\fl{e}}}
\end{bmatrix*}
}\leqslant
\sqrt{\overline{\varepsilon}^{\fl{r}}}
\tfrac{1}{2}(r_{i}^{\fl{e}\shortplus}+r_{i}^{\fl{e}\shortminus})\overline{p}_{ti} - x_{twhi}^{r^{\fl{e}}},
\\
&\left|\left[\left(P_{twh}-P_{tw(h-1)}\right)S_{t}\right]_{i}\mathbb{1}-\tfrac{1}{2}(r_{i}^{\fl{e}\shortplus}+r_{i}^{\fl{e}\shortminus})\overline{p}_{ti}\right|\leqslant z_{twhi}^{r^{\fl{e}}} + x_{twhi}^{r^{\fl{e}}}\\
&\tfrac{1}{2}(r_{i}^{\fl{e}\shortplus}+r_{i}^{\fl{e}\shortminus})\overline{p}_{ti}\geqslant x_{twhi}^{r^{\fl{e}}} \geqslant 0, z_{twhi}^{r^{\fl{e}}} \geqslant 0,
\end{aligned}
\right\}\forall i\in\llbracket N\rrbracket,\label{tractable_ref_ramp_ds}\\
&
\left.\begin{aligned}
&\norm{
\overline{\Sigma}\left[
r^{\fl{c}\shortminus}\circ\textstyle\sum_{\tau=1}^{t}\overline{Y}_{\tau}S_{\tau}+Y_{twh}S_{t}-Y_{tw(h-1)}S_{t}
\right]_{i}^{\top}
}\\
&\quad\leqslant
\tfrac{1}{\sqrt{(1-\overline{\varepsilon}^{\fl{g}})/\overline{\varepsilon}^{\fl{r}}}}\left[
r^{\fl{c}\shortminus}\circ\textstyle\sum_{\tau=1}^{t}\overline{Y}_{\tau}S_{\tau}+Y_{twh}S_{t}-Y_{tw(h-1)}S_{t}
\right]_{i}\mathbb{1},\\
&\norm{
\overline{\Sigma}\left[
r^{\fl{c}\shortplus}\circ\textstyle\sum_{\tau=1}^{t}\overline{Y}_{\tau}S_{\tau}-Y_{twh}S_{t}+Y_{tw(h-1)}S_{t}
\right]_{i}^{\top}
}\\
&\quad\leqslant
\tfrac{1}{\sqrt{(1-\overline{\varepsilon}^{\fl{g}})/\overline{\varepsilon}^{\fl{r}}}}\left[
r^{\fl{c}\shortplus}\circ\textstyle\sum_{\tau=1}^{t}\overline{Y}_{\tau}S_{\tau}-Y_{twh}S_{t}+Y_{tw(h-1)}S_{t}
\right]_{i}\mathbb{1},
\end{aligned}
\right\}\forall i\in\llbracket N\rrbracket,\label{tractable_ref_ramp_ss}\\
&\left.\begin{aligned}
&\norm{\overline{\Sigma}\left[\textstyle\sum_{\tau=1}^{t}\overline{\Theta}_{\tau wh}S_{\tau} - \Theta_{twh}S_{t}\right]_{i}^{\top}}\leqslant \tfrac{1}{\sqrt{(1-\overline{\varepsilon}^{\fl{s}})/\overline{\varepsilon}^{\fl{s}}}}\left[\textstyle\sum_{\tau=1}^{t}\overline{\Theta}_{\tau wh}S_{\tau} - \Theta_{twh}S_{t}\right]_{i}\mathbb{1},\\
&\norm{\overline{\Sigma}\left[\Theta_{twh}S_{t}\right]_{i}^{\top}}\leqslant \tfrac{1}{\sqrt{(1-\overline{\varepsilon}^{\fl{s}})/\overline{\varepsilon}^{\fl{s}}}}\left[\Theta_{twh}S_{t}\right]_{i}\mathbb{1},\\
&\norm{\overline{\Sigma}\left[\textstyle\sum_{\tau=1}^{t}\overline{\Phi}_{\tau wh}S_{\tau} - \Phi_{twh}^{\shortplus}S_{t}\right]_{i}^{\top}}\leqslant \tfrac{1}{\sqrt{(1-\overline{\varepsilon}^{\fl{s}})/\overline{\varepsilon}^{\fl{s}}}}\left[\textstyle\sum_{\tau=1}^{t}\overline{\Phi}_{\tau wh}S_{\tau} - \Phi_{twh}^{\shortplus}S_{t}\right]_{i}\mathbb{1},\\
&\norm{\overline{\Sigma}\left[\Phi_{twh}^{\shortplus}S_{t}\right]_{i}^{\top}}\leqslant \tfrac{1}{\sqrt{(1-\overline{\varepsilon}^{\fl{s}})/\overline{\varepsilon}^{\fl{s}}}}\left[\Phi_{twh}^{\shortplus}S_{t}\right]_{i}\mathbb{1},\\
&\norm{\overline{\Sigma}\left[\textstyle\sum_{\tau=1}^{t}\overline{\Phi}_{\tau wh}S_{\tau} - \Phi_{twh}^{\shortminus}S_{t}\right]_{i}^{\top}}\leqslant \tfrac{1}{\sqrt{(1-\overline{\varepsilon}^{\fl{s}})/\overline{\varepsilon}^{\fl{s}}}}\left[\textstyle\sum_{\tau=1}^{t}\overline{\Phi}_{\tau wh}S_{\tau} - \Phi_{twh}^{\shortminus}S_{t}\right]_{i}\mathbb{1},\\
&\norm{\overline{\Sigma}\left[\Phi_{twh}^{\shortminus}S_{t}\right]_{i}^{\top}}\leqslant \tfrac{1}{\sqrt{(1-\overline{\varepsilon}^{\fl{s}})/\overline{\varepsilon}^{\fl{s}}}}\left[\Phi_{twh}^{\shortminus}S_{t}\right]_{i}\mathbb{1},\\
&\norm{\overline{\Sigma}\left[\textstyle\sum_{\tau=1}^{t}\overline{\Phi}_{\tau wh}S_{\tau} - (\Phi_{twh}^{\shortplus} + \Phi_{twh}^{\shortminus})S_{t}\right]_{i}^{\top}}\leqslant \tfrac{1}{\sqrt{(1-\overline{\varepsilon}^{\fl{s}})/\overline{\varepsilon}^{\fl{s}}}}\left[\textstyle\sum_{\tau=1}^{t}\overline{\Phi}_{\tau wh}S_{\tau} - (\Phi_{twh}^{\shortplus} + \Phi_{twh}^{\shortminus})S_{t}\right]_{i}\mathbb{1}\\
\end{aligned}
\right\}\forall i\in\llbracket N\rrbracket,\label{tractable_ref_stor_ss}
\end{align}
\end{subequations}
in variables $\overline{\mathcal{V}} = \{\overline{Y},\overline{\Theta},\overline{\Phi},P,Y,\Theta,\Phi^{\mydot},z^{\mydot},x^{\mydot}\}$, plus emission and investment limits as in \eqref{eq:single_sided_ref} and \eqref{cc_inv_ref_1}--\eqref{cc_inv_ref_3}, respectively. Here, the double-sided entries in power flow constraint \eqref{cc_flow} are reformulated into \eqref{tractable_ref_flow} with $\overline{\varepsilon}^{\fl{f}}=\varepsilon^{\fl{f}}/E$. The double-sided existing generation limits in \eqref{cc_gen} are reformulated into \eqref{tractable_ref_gen_ds}, and single-sided candidate generation limits in \eqref{cc_gen} are reformulated into \eqref{tractable_ref_gen_ss}, while fixing $\overline{\varepsilon}^{\fl{g}}=\varepsilon^{\fl{g}}/(3N)$. Similarly, the double-sided ramping limits on existing generation in \eqref{cc_ramp} are reformulated into \eqref{tractable_ref_ramp_ds} and the single-sided ramping limits on candidate generation in \eqref{cc_ramp} are reformulated into \eqref{tractable_ref_ramp_ss}. Here, we set $\overline{\varepsilon}^{\fl{r}}=\varepsilon^{\fl{r}}/(3N)$. All entries in the operational storage constraint \eqref{cc_stor} are single-sided and reformulated into \eqref{tractable_ref_stor_ss} with $\overline{\varepsilon}^{\fl{s}}=\varepsilon^{\fl{s}}/(7N)$. 

\subsection{The Dual Stochastic Problem Formulation}\label{app:sto_dual}
The dual problem of the chance-constrained program \eqref{prog:sto} takes the following  form:
\begin{subequations}\label{prog:sto_dual_intr}
\begin{align}
\maximize{\lambda}\quad&\mathbb{E}\Bigg[\sum_{t=1}^{T}\Bigg[\sum_{w=1}^{W}\Bigg\langle\sum_{h=1}^{H}\Bigg(
\mathbb{1}^{\top}\text{diag}\left[k_{twh}^{\ell}\right]\ell_{t}(\boldsymbol{\xi}^{t})\lambda_{twh}^{b}(\boldsymbol{\xi}^{t})  
-\left(F\text{diag}\left[k_{twh}^{\ell}\right]\ell_{t}(\boldsymbol{\xi}^{t}) + \overline{f}\right)^{\top}\lambda_{twh}^{\overline{f}}(\boldsymbol{\xi}^{t})
\nonumber\\&\quad\quad\quad
+\left(F\text{diag}\left[k_{twh}^{\ell}\right]\ell_{t}(\boldsymbol{\xi}^{t}) - \overline{f}\right)^{\top}\lambda_{twh}^{\underline{f}}(\boldsymbol{\xi}^{t})
-\left(\text{diag}\left[k_{twh}^{\fl{e}}\right]\overline{p}_{t}\right)^{\top}\lambda_{twh}^{p}(\boldsymbol{\xi}^{t})\Bigg)\Bigg\rangle
\nonumber\\&\quad\quad\quad
-\overline{e}_{t}(\boldsymbol{\xi}^{t})\lambda_{t}^{e}(\boldsymbol{\xi}^{t}) 
-\overline{y}_{t}^{\fl{max}\top}\lambda_{t}^{\overline{y}}(\boldsymbol{\xi}^{t}) 
-\overline{\varphi}_{t}^{\fl{max}\top}\lambda_{t}^{\overline{\varphi}}(\boldsymbol{\xi}^{t}) 
-\overline{\vartheta}_{t}^{\fl{max}\top}\lambda_{t}^{\overline{\vartheta}}(\boldsymbol{\xi}^{t})\Bigg]
\nonumber\\&
-\sum_{t=1}^{T}\sum_{w=1}^{W}\sum_{h=2}^{H}\left(
\left(\text{diag}\left[r^{\fl{e}\shortminus}\right]\overline{p}_{t}\right)^{\top}\lambda_{twh}^{r^{\fl{e}\shortminus}}(\boldsymbol{\xi}^{t})
+\left(\text{diag}\left[r^{\fl{e}\shortplus}\right]\overline{p}_{t}\right)^{\top}\lambda_{twh}^{r^{\fl{e}\shortplus}}(\boldsymbol{\xi}^{t})\right)\Bigg]
\\
\st\quad
&\mathbb{P}\!\!\left[
\begin{aligned}
& q_{t}^{\fl{g}}(\boldsymbol{\xi}^{t})+\sum_{\tau=t}^{T}o_{\tau}^{\fl{g}}\leqslant
\sum_{\tau=t}^{T}
\sum_{w=1}^{W}
\Bigg(
\sum_{h=1}^{H}
\text{diag}\left[k_{twh}^{\fl{c}}\right]^{\top}\lambda_{\tau wh}^{y}(\boldsymbol{\xi}^{\tau}) \nonumber\\
&\quad\quad\quad\quad\quad\quad\;\;+ 
\sum_{h=2}^{H}
\big(\text{diag}\left[r^{\fl{c}\shortminus}\right]^{\top}\lambda_{\tau wh}^{r^{\fl{c}\shortminus}}(\boldsymbol{\xi}^{\tau})
+\text{diag}\left[r^{\fl{c}\shortplus}\right]^{\top}\lambda_{\tau wh}^{r^{\fl{c}\shortplus}}(\boldsymbol{\xi}^{\tau})\big)
\Bigg)
- \lambda_{t}^{\overline{y}}(\boldsymbol{\xi}^{t})\\
\end{aligned}
\right]\!\!\geqslant\!1\!-\!\varepsilon^{\overline{y}}\\
&\mathbb{P}\!\!\left[
\begin{aligned}
& q_{t}^{\fl{s}}(\boldsymbol{\xi}^{t})+\sum_{\tau=t}^{T}o_{\tau}^{\fl{s}}\leqslant
\sum_{\tau=t}^{T}
\sum_{w=1}^{W}
\sum_{h=1}^{H}
\lambda_{\tau wh}^{\vartheta}(\boldsymbol{\xi}^{\tau})
- \lambda_{t}^{\overline{\vartheta}}(\boldsymbol{\xi}^{t})\\
\end{aligned}
\right]\!\!\geqslant\!1\!-\!\varepsilon^{\overline{\vartheta}}\\
&\mathbb{P}\!\!\left[
\begin{aligned}
& q_{t}^{\fl{p}}(\boldsymbol{\xi}^{t})+\sum_{\tau=t}^{T}o_{\tau}^{\fl{p}}\leqslant
\sum_{\tau=t}^{T}
\sum_{w=1}^{W}
\sum_{h=1}^{H}
\left(\lambda_{\tau wh}^{\varphi^{\shortplus}}(\boldsymbol{\xi}^{\tau})
+\lambda_{\tau wh}^{\varphi^{\shortminus}}(\boldsymbol{\xi}^{\tau})
+\lambda_{\tau wh}^{\varphi}(\boldsymbol{\xi}^{\tau})
\right)
- \lambda_{t}^{\overline{\vartheta}}(\boldsymbol{\xi}^{t})\\
\end{aligned}
\right]\!\!\geqslant\!1\!-\!\varepsilon^{\varphi}\\
&\mathbb{P}\!\!\left[
\begin{aligned}
& \omega_{w}c_{t}^{\fl{e}}(\boldsymbol{\xi}^{t})\leqslant
\lambda_{twh}^{b}(\boldsymbol{\xi}^{t})\mathbb{1}
+F^{\top}\left(\lambda_{twh}^{\underline{f}}(\boldsymbol{\xi}^{t}) - \lambda_{twh}^{\overline{f}}(\boldsymbol{\xi}^{t})\right)
-\lambda_{twh}^{p}(\boldsymbol{\xi}^{t})\nonumber\\
&\quad\quad\quad\quad+\underbrace{\lambda_{twh}^{r^{\fl{e}\shortminus}}(\boldsymbol{\xi}^{t}) - \lambda_{twh}^{r^{\fl{e}\shortplus}}(\boldsymbol{\xi}^{t})}_{h\geqslant2}
+\underbrace{\lambda_{tw(h+1)}^{r^{\fl{e}\shortplus}}(\boldsymbol{\xi}^{t}) - \lambda_{tw(h+1)}^{r^{\fl{e}\shortminus}}(\boldsymbol{\xi}^{t})}_{h<H}
-\omega_{w}\lambda_{t}^{e}(\boldsymbol{\xi}^{t})e^{\fl{e}}
\end{aligned}
\right]\!\!\geqslant\!1\!-\!\varepsilon^{p}\\
&\mathbb{P}\!\!\left[
\begin{aligned}
& \omega_{w}c_{t}^{\fl{c}}(\boldsymbol{\xi}^{t})\leqslant
\lambda_{twh}^{b}(\boldsymbol{\xi}^{t})\mathbb{1}
+F^{\top}\left(\lambda_{twh}^{\underline{f}}(\boldsymbol{\xi}^{t}) - \lambda_{twh}^{\overline{f}}(\boldsymbol{\xi}^{t})\right)
-\lambda_{twh}^{y}(\boldsymbol{\xi}^{t})\\
&\quad\quad\quad\quad+\underbrace{\lambda_{twh}^{r^{\fl{c}\shortminus}}(\boldsymbol{\xi}^{t}) - \lambda_{twh}^{r^{\fl{c}\shortplus}}(\boldsymbol{\xi}^{t})}_{h\geqslant2}
+\underbrace{\lambda_{tw(h+1)}^{r^{\fl{c}\shortplus}}(\boldsymbol{\xi}^{t}) - \lambda_{tw(h+1)}^{r^{\fl{c}\shortminus}}(\boldsymbol{\xi}^{t})}_{h<H}
-\omega_{w}\lambda_{t}^{e}(\boldsymbol{\xi}^{t})e^{\fl{c}}
\end{aligned}
\right]\!\!\geqslant\!1\!-\!\varepsilon^{y}\\
&\mathbb{P}\!\!\left[
\begin{aligned}
&\mathbb{0}\leqslant
\lambda_{twh}^{b}(\boldsymbol{\xi}^{t})\mathbb{1}
+F^{\top}\left(\lambda_{twh}^{\underline{f}}(\boldsymbol{\xi}^{t}) - \lambda_{twh}^{\overline{f}}(\boldsymbol{\xi}^{t})\right) \\
&\quad\!\underbrace{+ \lambda_{twh}^{s}(\boldsymbol{\xi}^{t})\tfrac{1}{\eta^{\shortminus}}}_{h\geqslant2}
\underbrace{+ \lambda_{tw}^{s1}(\boldsymbol{\xi}^{t})\tfrac{1}{\eta^{\shortminus}}}_{h=1}
-\lambda_{twh}^{\varphi^{\shortminus}}(\boldsymbol{\xi}^{t})
-\lambda_{twh}^{\varphi}(\boldsymbol{\xi}^{t})\\
\end{aligned}
\right]\!\!\geqslant\!1\!-\!\varepsilon^{\varphi^{\shortminus}}\\
&\mathbb{P}\!\!\left[
\begin{aligned}
& \mathbb{0}\geqslant
\lambda_{twh}^{b}(\boldsymbol{\xi}^{t})\mathbb{1}
+F^{\top}\left(\lambda_{twh}^{\underline{f}}(\boldsymbol{\xi}^{t}) - \lambda_{twh}^{\overline{f}}(\boldsymbol{\xi}^{t})\right) \\
&\quad\!\underbrace{+ \lambda_{twh}^{s}(\boldsymbol{\xi}^{t})\eta^{\shortplus}}_{h\geqslant2}
\underbrace{+ \lambda_{tw}^{s1}(\boldsymbol{\xi}^{t})\eta^{\shortplus}}_{h=1}
+\lambda_{twh}^{\varphi^{\shortplus}}(\boldsymbol{\xi}^{t})
+\lambda_{twh}^{\varphi}(\boldsymbol{\xi}^{t})
\end{aligned}
\right]\!\!\geqslant\!1\!-\!\varepsilon^{\varphi^{\shortplus}}\\
&\mathbb{P}\!\!\left[
\begin{aligned}
\mathbb{0}\leqslant
\underbrace{\lambda_{tw}^{s1}(\boldsymbol{\xi}^{t})}_{h=1}
\underbrace{+ \lambda_{twh}^{s}(\boldsymbol{\xi}^{t})}_{h\geqslant2}
\underbrace{- \lambda_{tw(h+1)}^{s}(\boldsymbol{\xi}^{t})}_{h<H}
-\lambda_{twh}^{\vartheta}(\boldsymbol{\xi}^{t})
\end{aligned}
\right]\!\!\geqslant\!1\!-\!\varepsilon^{\vartheta}\\
&\forall\mathbb{P}\in\mathcal{P},\;\forall t\in\llbracket T\rrbracket,\;\forall w\in\llbracket W\rrbracket,\;\forall h\in\llbracket H\rrbracket,\nonumber
\end{align}
\end{subequations}
where the optimization variables are denoted by the Greek letter $\lambda$. Here, the expected value of the dual objective function is maximized subject to the series of joint chance constraints, where the subscript $\odot$ in  $\varepsilon^{\mydot}$ denotes the primal variables to which the dual constraints correspond. Note, that the tractable second-order cone programming form of problem \eqref{prog:sto_dual_intr} is achieved similarly to problem \eqref{prog:full_tract_ref} and omitted in the interest of space.


\endgroup
\end{document}